\documentclass[a4paper,11pt]{article}
\usepackage[T1]{fontenc}
\usepackage[utf8]{inputenc}
\usepackage[english]{babel}
\usepackage[margin=0.8in]{geometry}
\usepackage{microtype}
\usepackage{braket}
\usepackage{aliascnt}

\usepackage{amsmath}
\usepackage{amssymb}
\usepackage{amsthm}
\usepackage{mathtools}
\usepackage{mathrsfs}
\usepackage{xcolor}
\usepackage[hyphens]{url}
\usepackage{enumerate}
\usepackage{tikz}
\usepackage[maxbibnames=99]{biblatex}
\usepackage{csquotes}
\usetikzlibrary{patterns}
\usepackage{esint}
\usepackage{mleftright}

\usepackage{csquotes}
\usepackage[maxbibnames=99]{biblatex}
\usepackage{xcolor}

\definecolor{OmegaCol}{RGB}{255,185,110}   
\definecolor{SigmaCol}{RGB}{170,215,255}   
\definecolor{TauCol}{RGB}{95,120,85}

\usetikzlibrary{shapes.geometric,arrows.meta,positioning}
\usetikzlibrary{decorations.pathreplacing}
\usetikzlibrary{decorations.markings}
\usetikzlibrary{calc}

\usepackage{hyperref}
\hypersetup{
    colorlinks=true,
    linkcolor=magenta,
    citecolor=cyan,
    urlcolor=magenta,
    breaklinks=true,
}

\addto\extrasitalian{%
}
\addto\extrasenglish{
    
}

\theoremstyle{plain}
\newtheorem{teor}{Theorem}
\numberwithin{teor}{section}
\numberwithin{equation}{section}

\theoremstyle{definition}
\newaliascnt{defi}{teor}
\newtheorem{defi}[defi]{Definition}
\aliascntresetthe{defi}

\theoremstyle{plain}
\newaliascnt{lemma}{teor}
\newtheorem{lemma}[lemma]{Lemma}
\aliascntresetthe{lemma}

\theoremstyle{plain}
\newaliascnt{prop}{teor}
\newtheorem{prop}[prop]{Proposition}
\aliascntresetthe{prop}

\theoremstyle{plain}
\newaliascnt{conjecture}{teor}

\aliascntresetthe{conjecture}
 
\theoremstyle{plain}
\newaliascnt{cor}{teor}
\newtheorem{cor}[cor]{Corollary}
\aliascntresetthe{cor}
 
\theoremstyle{definition}
\newaliascnt{ex}{teor}

\aliascntresetthe{ex}
 
\theoremstyle{definition}
\newaliascnt{oss}{teor}
\newtheorem{oss}[oss]{Remark}
\aliascntresetthe{oss}
 
\theoremstyle{plain}

\usepackage{thm-restate}
\newcommand{\Hlabel}[1]{
  \hypertarget{H#1}{}
  \textnormal{\textcolor{magenta}{{H#1}}}
}

\newcommand{\Alabel}[1]{
  \hypertarget{C#1}{}
  \textnormal{\textcolor{magenta}{{C#1}}}
}

\newcommand{\Href}[1]{\hyperlink{#1}{\textnormal{\textcolor{magenta}{{#1}}}}}

\DeclarePairedDelimiter{\abs}{\lvert}{\rvert}
\DeclarePairedDelimiter{\norma}{\lVert}{\rVert}

\DeclareMathOperator{\spn}{span}

\newcommand{\R}{\mathbb{R}}

\newcommand{\N}{\mathbb{N}}

\newcommand{\Hn}{\mathcal{H}^{n-1}}

\newcommand{\eps}{\varepsilon}

\newcommand{\Om}{\Omega}

\DeclareMathOperator{\divv}{div}

\newcommand{\Addresses}{{ 
 \bigskip 
 \footnotesize 

 \textsc{Dipartimento di Matematica ``Federigo Enriques'', Universit\'a degli Studi di Milano La Statale, Via Saldini 50 20123 Milano Italia.}\par\nopagebreak 
 
 \medskip 
 
 \textit{E-mail address}, E.~Cristoforoni: \texttt{emanuele.cristoforoni@unimi.it} 

  \bigskip 

\textsc{Mathematical and Physical Sciences for Advanced Materials and Technologies, Scuola Superiore Meridionale, Largo San Marcellino 10, 80138, Napoli, Italy.}\par\nopagebreak 
 
 \medskip 

  \textit{E-mail address}, 
  F.~Villone: \texttt{f.villone@ssmeridionale.it} 
}}

\title{Remarks on the reinforcement of the spectrum of an elliptic problem with Robin boundary condition}
\author{Emanuele Cristoforoni, Federico Villone*}
\date{}

\begin{document}
\maketitle
\begin{abstract}
We investigate the spectral properties of a differential elliptic operator on $H^1(\bar{\Om}\cup \Sigma_\eps)$, where $\Om$ is a smooth domain surrounded by a layer $\Sigma_\eps$. The thickness of the layer is given by $\eps h$, where $h$ is a positive function defined on the boundary $\partial \Om$ and $\eps$ is the ellipticity constant of the operator in $\Sigma_\eps$. We prove that, in the limit for $\eps$ going to $0$, the spectrum converges to the spectrum of a differential elliptic operator in $H^1(\Om)$, and we investigate a first-order asymptotic development.\smallskip

    
\end{abstract}

\begin{center}
\begin{minipage}{10cm}
\small
\tableofcontents
\end{minipage}
\end{center}

\section{Introduction}
Let $\Omega\subset\R^n$ be a bounded, open set with $C^{1,1}$ boundary, let $\nu_0$ be the unit, outer normal vector to the boundary of $\partial \Omega$.
Let $h:\partial\Om\to \R$ be a strictly
positive Lipschitz-continuous function. For all $\eps>0$, consider 
\[\Sigma_\eps:=\Set{\sigma+t\nu_0(\sigma)| \begin{gathered}
\sigma\in\partial\Om, \\
0<t<\eps h(\sigma)
\end{gathered}}.\]

We denote by 
\[\Om_\eps=\overline{\Om}\cup \Sigma_\eps,\]
and by $\nu_\eps$ the unit, outer normal to its boundary (see \autoref{figura omeps}).

\begin{figure}[htbp]
\centering

 \begin{tikzpicture}[scale=0.66]

\definecolor{OmegaCol}{RGB}{255,185,110}   
\definecolor{SigmaCol}{RGB}{170,215,255}   
\definecolor{TauCol}{RGB}{95,120,85}


\begin{scope}[scale=0.67, rotate=90, shift={(5,0)}]

\fill[OmegaCol]
  plot [smooth cycle, tension=0.9]
  coordinates {
    (-5,0)
    (-3.8,2.6)
    (0,3)
    (3.8,2.6)
    (5,0)
    (3.8,-2.6)
    (0,-3)
    (-3.8,-2.6)
  };

\draw[thick]
  plot [smooth cycle, tension=0.9]
  coordinates {
    (-5,0)
    (-3.8,2.6)
    (0,3)
    (3.8,2.6)
    (5,0)
    (3.8,-2.6)
    (0,-3)
    (-3.8,-2.6)
  };

\fill[SigmaCol]
  plot [smooth cycle, tension=0.9]
  coordinates {
    (-4.9,0)
    (-3.65,2.55)
    (0,2.85)
    (3.65,2.55)
    (4.9,0)
    (3.65,-2.55)
    (0,-2.85)
    (-3.65,-2.55)
  };

  \draw[dashed] (4,-1.8) rectangle (4.7,-2.5);
\end{scope}

\draw[dashed] (1.7,6.3) -- (6.2,4.3);


\begin{scope}[scale=1.8, shift={(-3.5,-0.2)}]

\begin{scope}
\clip (6.2,-0.1) rectangle (10.2,3.9);

\begin{scope}[shift={(8.8,2.0)}, scale=2.7]


\fill[OmegaCol]
  plot [smooth cycle, tension=1]
  coordinates {
    (-2.4,-3)   
    (-2.4,1.0)  
    (-1.0,0.5)
    (0,0.0)
    (1.0,-1.0)
    (2.4,-2.6)
    (2.7,-1)    
  };
  
\fill[SigmaCol]
  plot [smooth cycle, tension=1]
  coordinates {
    (-2.3,0.5)   
    (-1.0,0.0)   
    (0,-0.5)     
    (1.0,-1.5)   
    (2.3,-3.0)   
    (2.3,-5)     
    (-2.3,-5)    
  };

\draw[thick]
  plot [smooth, tension=1]
  coordinates {
    (-2.3,1.0)
    (-1.0,0.5)
    (0,0.0)
    (1.0,-1.0)
    (2.3,-2.5)   
  };

\draw[->,  line width=0.7pt] (-0.26,-0.32) -- (-0.1,-0.123)
  node[pos=0.4, above=0.01pt, sloped] { $\nu_0$} ;

\draw[decorate, decoration={brace, mirror, amplitude=8pt}, line width=0.8pt]
  (-0.26,-0.32) -- (-0.0,0.0)
  node[pos=0.55, sloped, below=7pt] { $\varepsilon h$};


\end{scope}
\end{scope}
\node at (6.85,2.45) {$\Sigma_\eps$} ;
\node at (6.6,0.6) {${\Om}$} ;
\draw (6.2,-0.1) rectangle (10.2,3.9);

\node at (8,1) {${\sigma}$};



\end{scope}

\end{tikzpicture}
\caption{In blue the set $\Omega$, in orange the  layer $\Sigma_\eps$.}
\label{figura omeps}
\end{figure}

We remark that, by the regularity of $\Omega$, there exists a neighbourhood $\Gamma_{d_0}$ of $\partial\Omega$ such that every $x\in \Gamma_{d_0}$ admits a unique metric projection, $\sigma(x)$, on $\partial\Omega$, moreover, every $x\in \Gamma_{d_0}$ can be uniquely written as
\[x=\sigma(x)+d(x)\nu_0(\sigma(x)),\]
where $d(x)$ denotes the signed distance of $x$ from $\partial\Omega$. Hence, the mapping $(\sigma,t)\mapsto \sigma+t\nu_0(\sigma)$ is invertible on $\Sigma_\eps$ for $\eps$ sufficiently small.\medskip

Fix $\beta>0$. For every $\eps>0$ we consider the following eigenvalue problem with Robin boundary condition.
\begin{equation}\label{eq forte autofunz}
\begin{cases}
-\Delta u = \lambda u & \text{in } \Omega, \\[5 pt]
-\eps\Delta u = \lambda u  & \text{in } \Sigma_\eps, \\[5 pt]
u^-=u^+ & \text{on } \partial \Om, \\[5 pt]
\displaystyle\dfrac{\partial {{u}}^-}{\partial \nu_0} = \varepsilon \dfrac{\partial {u}^+}{\partial \nu_0} & \text{on } \partial \Omega,\\[10 pt]

\eps\displaystyle\dfrac{\partial u}{\partial \nu_\eps} + \beta u = 0 & \text{on } \partial \Omega_\eps, 
\end{cases}
\end{equation}
where $u^-$ and $u^+$ denote the traces of $u$ on $\partial\Omega$ from $\Omega$ and from $\Sigma_\eps$ respectively. By classical theory, the eigenvalue problem \eqref{eq forte autofunz} admits a discrete spectrum 
\[0<\lambda^1_{ \eps}(h)\le\lambda^2_{ \eps}(h)\le\dots\le\lambda^j_{ \eps}(h)\le\dots\to+\infty.\]

We study the asymptotic behaviour of $\lambda_{ \eps}^j(h)$ when $\eps$ approaches $0$. In particular, we will show that, for every $j\in\N$, $\lambda_{ \varepsilon}^j(h)$ converges to $\lambda^j(h)$, the $j$-th eigenvalue of the following limit problem
\begin{equation}\label{eq forte autoval limite}
\begin{cases}
-\Delta u = \lambda u & \text{in } \Omega, \\[5 pt]
\displaystyle\dfrac{\partial u}{\partial \nu_0} + \dfrac{\beta}{1+\beta h} u = 0 & \text{on } \partial \Omega,
\end{cases}
\end{equation}
and study the dependence of the limit eigenvalues on the function $h$. The limit of the solutions to elliptic problems of this type, and their properties, has been extensively studied in the case of the Dirichlet boundary condition in \cite{BCF80, F80, AB86, BBN17} (see also \cite{B88, BB19, HLL22, HLL24, AKK25, AKK25bis} ) and more recently for the Robin boundary condition in \cite{ZRWZ09, LWZZ12, DPNST21, DPO25} (see also \cite{ACNT24} and \cite{AC25}). In particular, in \cite{F80}, \cite{ZRWZ09} and \cite{DPO25}, the authors investigates asymptotic properties of the principal eigenvalue. We finally refer to \cite{CNT26} for a brief survey paper on the subject and on the related optimisation problems. \medskip

The interest in the eigenvalue problem \eqref{eq forte autofunz} is motivated by the study of thermal insulation. Indeed, if $\Omega$ represents a thermal conductor of thermal diffusivity $k_\Omega$ surrounded by a thin layer of highly insulated material $\Sigma_\eps$ of thermal diffusivity $k_\Sigma<<k_\Omega$, setting $\eps=k_\Sigma/k_\Omega$, we have that, assuming the outside temperature to be equal to zero, the temperature, $T$, in the insulated body $\Omega_\eps$ is a solution to the heat equation
\[\begin{cases}
     \dfrac{\partial T}{\partial t} -\Delta T = f &\text{ in } \Omega\times(0,+\infty),\\[10 pt]
     \dfrac{\partial T}{\partial t} -\eps\Delta T = f &\text{ in } \Sigma_\eps\times(0,+\infty),\\[10 pt]
       T^- = T^+ &\text{ on }\partial\Omega\times(0,+\infty),\\[10 pt]
     \dfrac{\partial T^-}{\partial \nu_0} = \eps  \dfrac{\partial T^+}{\partial \nu_0} &\text{ on }\partial\Omega\times(0,+\infty),\\[10 pt]
     \eps \dfrac{\partial T}{\partial \nu_\eps}+ \beta T=0 &\text{ on }\partial\Om_\eps\times(0,+\infty),\\[10pt]
     T(\cdot,0)=T_{0}&\text{ in }\Omega_\eps,
\end{cases}\]
where $f$ represents the heat source, $T_0$ is the initial temperature, and the Robin boundary condition, according to Newton's law of cooling, models the case in which the main mode of heat transfer with the environment is through convection. Hence, the temperature $T$ can be written as 
\[T(x,t)=T_\infty(x)+\sum_{i=1}^\infty c_i e^{-\lambda_\eps^i t}u_\eps^i(x),\]
where $(\lambda_\eps^i,u_\eps^i)$ are the solutions to the eigenvalue  problem \eqref{eq forte autofunz}, and $T_\infty$ is the stationary solution to the heat equation, that is, the solution to the Poisson problem
\[\begin{cases}
     -\Delta T_\infty = f &\text{ in } \Omega,\\[10 pt]
     -\eps\Delta T_\infty = f &\text{ in } \Sigma_\eps,\\[10 pt]
       T^-_\infty = T^+_\infty &\text{ on }\partial\Omega,\\[10 pt]
     \dfrac{\partial T^-_\infty}{\partial \nu_0} = \eps  \dfrac{\partial T^+_\infty}{\partial \nu_0} &\text{ on }\partial\Omega\,\\[10 pt]
     \eps \dfrac{\partial T_\infty}{\partial \nu_\eps}+ \beta T_\infty=0 &\text{ on }\partial\Om_\eps.
\end{cases}\]

\medskip

We now state the main results of the paper. In \autoref{spectr} we study the limit for $\eps$ going to zero of the eigenvalues $\lambda_\eps^j=\lambda_\eps^j(h)$ and prove the following theorem

\begin{teor}\label{cor3.6}
Let $\set{\lambda_\eps^j}$ and $\set{\lambda^j}$ denote the sequences of eigenvalues of problems \eqref{eq forte autofunz} and \eqref{eq forte autoval limite} respectively, counted with multiplicity. Fix $\{u_\eps^j\}$ be an orthonormal basis of eigenfunctions associated with problem \eqref{eq forte autofunz}. Then the following properties hold:
\begin{enumerate}[(i)]
    \item For every $j\in\mathbb{N}$, 
    \begin{equation}
        \lim_{\eps\to0^+}\lambda_\eps^j=\lambda^j.
    \end{equation}
    \item There exist $\set{\eps_k}$ a vanishing sequence and $\set{u^j}$ an orthonormal basis of eigenfunctions associated with problem \eqref{eq forte autoval limite} such that for every $j\in\mathbb{N}$
    \begin{equation}\label{convergenza estratta autofunzioni lambd}
        \lim_{k\to+\infty}\norma{u_{\eps_k}^j- u^j}_{L^2(\Om)}= 0,\quad \lim_{k\to+\infty}\norma{u_{\eps_k}^j}_{L^2(\Sigma_{\eps_k})}= 0.
    \end{equation}
  \item Assume that $\lambda$ is an eigenvalue of multiplicity $l+1$ for problem \eqref{eq forte autoval limite} such that \[\lambda=\lambda^j=\lambda^{j+1}=\dots=\lambda^{j+l}.\] 
  Then for all $w$ such that $w$ is an eigenfunction associated to $\lambda$ for \eqref{eq forte autoval limite}, there exists a family $\set{w_\eps}$ such that  $w_\eps\in \spn\{u_\eps^j,\dots ,u_\eps^{j+l}\}$ and
  \begin{equation}\label{conv ad autofunzione lambd}
      \lim_{\eps\to0^+}\norma{ w_\eps- w}_{L^2(\Om)}= 0, \quad \lim_{\eps\to0^+}\norma{ w_\eps}_{L^2(\Sigma_\eps)}= 0,
  \end{equation}
  
\end{enumerate}

\end{teor}

In \autoref{opt}, we study the dependence of the eigenvalues $\lambda^j(h)$ on the function $h$ and related optimisation problems. For every $m>0$, let
\[\mathcal{H}_m(\partial \Om):=\left\{h\in L^1(\partial\Om), \,h\geq 0:\, \int_{\partial\Om}  h\,d\Hn\leq m \right\},\]
here $\Hn$ is the $(n-1)$-dimensional Hausdorff measure on $\R^n$. As the measure of the set $\Sigma_\eps$ can be approximated as
\[\abs{\Sigma_\eps}=\eps\int_{\partial\Omega} h\,d\Hn+o(\eps),\]
the integral constraint on $h$ is an approximation of the measure constraint on $\Sigma_\eps$, one would naturally impose in the study of the dependence of the eigenvalues $\lambda_\eps^j$ on the shape of the layer $\Sigma_\eps$.  The main result of the section is the following theorem.
\begin{teor}\label{optlambda}
Let $m>0$. For all $j\in\mathbb{N}$ there exists $\bar{h}_j\in\mathcal{H}_m$ such that 
\[\lambda^j(\bar{h}_j)=\min_{h\in\mathcal{H}_m}\lambda^j(h) \quad \text{and}\quad\int_{\partial\Om}\bar{h}_j\,d\Hn=m.\]

Moreover, if we consider $F:\R^r\to \R$ a lower semicontinuous function and $j_1\,,\dots,\, j_r\in\mathbb{N}$, then there exists
\[\min_{h\in \mathcal{H}_m} F(\lambda^{j_1}(h),\,\lambda^{j_2}(h),\,\dots,\,\lambda^{j_r}(h))\]
\end{teor}

In \autoref{dev} we use the so-called \textit{lemma on small eigenvalues}, introduced by Colin de 
Verdière in \cite{CV86}, to study the limit of the quotient 
\[\dfrac{\lambda_\eps^j-\lambda^j}{\eps}.\]
Namely, for every $\lambda\in\R$, consider the bilinear form on $H^1(\Omega)$ 
\[Q_\lambda(u,v;h) =
\int_{\partial \Om}\dfrac{\beta Hh(2+\beta h)}{2(1+\beta h)^2}uv\,d\Hn -\lambda\int_{\partial\Om}\dfrac{h(3+3\beta h+\beta^2h^2)}{3(1+\beta h)^2}uv\,d\Hn,\]
where $H$ denotes the mean curvature of the boundary of $\Omega$. We then prove the following theorem

\begin{teor}\label{asymptotic}
Fix a positive function $h \in C^{1,1}(\Gamma_{d_0})$ such that $h(x) = h(\sigma(x))$. Then, for every $i\in\N$ the following limit exists
\[\lim_{\eps\to0^+}\dfrac{\lambda^i_\eps(h)-\lambda^i(h)}{\eps},\]
 and can be characterised as follows.
Let $l\in\N_0$, let $\lambda$ be an eigenvalue of \eqref{eq forte autoval limite} with multiplicity $l+1$,  let $i\in\N$ such that
\[\lambda=\lambda^i(h)=\dots=\lambda^{i+l}(h),\]
and denote by $E(\lambda)$ the associated eigenspace. for every $j=0,\dots,l$ denote by $\zeta^{j}_\lambda(h)$ the $(j+1)$-th eigenvalue of the bilinear form $Q_\lambda$ restricted to the $(l+1)$-dimensional space $E(\lambda)$.  Then for every $j=0,\dots,l$ we have
\begin{equation}\label{differeig}
    \lim_{\eps\to0^+}\dfrac{\lambda_\eps^{i+j}(h)-\lambda^{i+j}(h)}{\eps}= \zeta^{j}_\lambda(h).
\end{equation}
\end{teor}

We remark that, for the case of the principal eigenvalue with Dirichlet boundary condition (formally $\beta=+\infty$)  the first-order asymptotic development for the principal eigenvalue has already been proved in \cite{Y18}. Finally, in \autoref{Dirichlet}, we discuss the previous theorems in the case in which we replace the Robin boundary condition, in the original problem, with a Dirichlet one

\section{Convergence of the spectrum}\label{spectr}

To facilitate the analysis of the asymptotic behaviour of the eigenvalues, we will consider the resolvent operators. In this way, we can see both $\lambda_{\varepsilon}^j(h)$ and $\lambda^j(h)$ as the reciprocals of the eigenvalues of an operator. This reformulation allows for a clearer comparison of the spectra as $\eps$ approaches $0$.\medskip

For all $\eps>0$ and $f\in L^2(\Om_\eps)$ we consider $v_{\eps,h, f}=\mathcal{T}_{\eps,h} f\in H^1(\Om_\eps)$ the unique weak solution of
\[
\begin{cases}
-\Delta v_{\varepsilon,h,f} = f & \text{in } \Omega, \\[5 pt]
-\eps\Delta v_{\varepsilon,h,f} = f & \text{in } \Sigma_\eps, \\[5pt]
v_{\eps, h,f}^-=v_{\eps,h,f}^+&\text{on } \partial \Om, \\[5 pt]
\displaystyle\dfrac{\partial {v_{\eps, h,f}^-}}{\partial \nu_0} = \varepsilon \dfrac{\partial v_{\eps, h,f}^+}{\partial \nu_0} & \text{on } \partial \Omega,\\[10 pt]
\eps\displaystyle\dfrac{\partial v_{\eps, h, f}}{\partial \nu_\eps} + \beta v_{\varepsilon,h, f} = 0 & \text{on } \partial \Omega_\varepsilon. 
\end{cases}
\]
That is the unique solution of
\begin{equation}\label{eq forte}\int_{\Omega} \nabla v_{\eps,h,f}\nabla \varphi\,dx+\eps\int_{\Sigma_\eps} \nabla v_{\eps,h,f}\nabla \varphi\,dx+\beta\int_{\partial\Omega_\eps}v_{\eps,h,f}\varphi\,d\Hn=\int_{\partial\Omega_\eps}f\varphi\,dx,\end{equation}
for every $\varphi\in H^1(\Omega_\eps)$.
Moreover $v_{\eps,h,f}$ is also the unique minimiser in $H^1(\Om_\eps)$ of the functional
\begin{equation}\label{funzionale Gepsf}
    G_{\eps,f}(w):=\int_{\Om}\abs{\nabla w}^2 \,dx+\eps\int_{\Sigma_\eps}\abs{\nabla w}^2 \,dx+\beta \int_{\partial\Om_\eps} w^2\,d\Hn- 2\int_{\Om_\eps} w f\,dx.
\end{equation}
In the following, we will use the notation 
\[
a(t) = \begin{cases} 
\eps & \text{if } t > 0, \\[5 pt]
1 & \text{if } t \leq 0,
\end{cases}
\]
so that, if $d$ denotes the signed distance form $\Omega$, the functional $G_{\eps,f}$ can be more compactly written as
\[ G_{\eps,f}(w)=\int_{\Om_\eps} a(d)\abs{\nabla w}^2 \,dx+\beta \int_{\partial\Om_\eps} w^2\,d\Hn- 2\int_{\Om_\eps} w f\,dx.\]

Similarly, for all $f\in L^2(\Om)$ we consider $v_{h,f}=\mathcal{T}_hf\in H^1(\Om)$ the unique weak solution of the boundary value problem 
\[
\begin{cases}
-\Delta v_{h,f} = f & \text{in } \Omega, \\[5 pt]
\displaystyle\dfrac{\partial v_{h,f}}{\partial \nu_0} + \dfrac{\beta}{1+\beta h} v_{h,f} = 0 & \text{on } \partial \Omega.
\end{cases}
\]
That is the unique solution of
\begin{equation}\label{eq forte limite}\int_\Omega\nabla v_{h,f}\nabla\varphi\,dx+\beta\int_{\partial\Omega}\dfrac{v_{h,f}\varphi}{1+\beta h}\,d\Hn=\int_\Om f\varphi\,dx,\end{equation}
for every $\varphi\in H^1(\Omega)$.
The solution $v_{h,f}$ is also the unique minimiser in $H^1(\Om)$ of the functional
\begin{equation}\label{G 0 f}
G_{f}(w):=\int_{\Om}|\nabla w|^2 \,dx+\beta\int_{\partial\Om} \dfrac{w^2}{1+\beta h} \,d\Hn- 2\int_{\Om} w f\,dx.
\end{equation}

For $\eps>0$, the operator $ \mathcal{T}_{\eps,h} : L^2(\Omega_\varepsilon) \to L^2(\Omega_\varepsilon) $, as previously defined, is positive, linear, continuous, compact, and self-adjoint. Similarly, $ \mathcal{T}_h : L^2(\Omega) \to L^2(\Omega) $ satisfies the same properties.
For the sake of simplicity, whenever there is no ambiguity, we will drop the explicit dependence on $h$.\medskip

We now state a Poincarè-type inequality, which will be useful in the following. We refer, for instance,  to \cite[Lemma 5.2]{AC25} for the proof. 
\begin{lemma}\label{lemma:poinc}
Let $h\colon\partial\Omega\to\R$ be a positive  Lipschitz function . Then there exist positive constants $d_0=d_0(\Omega)$, and $C_p=C_p(\Omega,\norma{h}_{C^{0,1}},\beta)$ such that if
\begin{equation}
    \eps\norma{h}_{\infty}\le d_0,
\end{equation}
then for every $w\in H^1(\Omega_\eps)$
\begin{equation}\label{eq:epspoinc}\int_{\Omega_\eps} w^2\,dx\le C_p\left[\int_{\Omega_\eps} a(d) \abs{\nabla w}^2\,dx+\beta\int_{\partial\Omega_\eps} w^2\,d\Hn\right],\end{equation}
and
\begin{equation}\label{eq:epspoinc2}\int_{\Sigma_\eps} w^2\,dx\le \eps C_p\left[\eps\int_{\Sigma_\eps} \abs{\nabla w}^2\,dx+\beta\int_{\partial\Omega_\eps} w^2\,d\Hn\right].\end{equation}
\end{lemma}

Our aim is to prove the convergence of the spectrum of $\mathcal{T}_{\eps,h}$ to the spectrum of $\mathcal{T}_h$. To do so, we will apply the theory developed in \cite{JOS89}. 

\begin{teor}[\cite{JOS89}, Chapter III \S 1.2]\label{convergenza spettri astratto}
Let ${H}_\varepsilon$, ${H}_0$ be separable Hilbert spaces with scalar products $(\cdot, \cdot)_\varepsilon$, $(\cdot, \cdot)_0$ and norms $\| \cdot \|_\varepsilon$, $\| \cdot \|_0$ respectively. Consider  $\mathcal{A}_\varepsilon:{H}_\eps\to{H}_\eps $ and $\mathcal{A}_0:{H}_0\to{H}_0$ continuous linear operators such that for all $w\in{H}_0$, $\mathcal{A}_0 w\in \mathcal{V}$, where $\mathcal{V}$ is a linear subspace of ${H}_0$.
Assume the following:

\begin{enumerate}
  \item[\Hlabel{1}.] There exist continuous linear operators $R_\eps: {H}_0 \to {H}_\eps$ such that for any $f \in \mathcal{V}$ we have 
  \[
  \lim_{\eps\to 0}(R_\eps f, R_\eps f)_\eps \to (f,f)_0,
  \]

  \item[\Hlabel{2}.] The operators $\mathcal{A}_\eps$, $\mathcal{A}_0$ are positive, compact, and self-adjoint in ${H}_\eps$ and ${H}_0$ respectively. Moreover
  \[
  \sup_\eps \| \mathcal{A}_\eps \|_{\mathcal{L}({H}_\eps)} < \infty.
  \]

  \item[\Hlabel{3}.] For any $f \in \mathcal{V}$ we have
  \[
  \lim_{\eps\to 0}\| \mathcal{A}_\eps R_\eps f - R_\eps \mathcal{A}_0 f \|_\eps=0.
  \]

  \item[\Hlabel{4}.] The family of operators $\Set{\mathcal{A}_\eps}$ is uniformly compact in the following sense. For any family $\Set{f_\eps}$ with $f_\eps \in {H}_\eps$ for every $\eps>0$, such that $\sup_\eps \| f^\eps \|_\eps < \infty$, there exists a sequence $\set{f_{\eps_k}}$ such that for some $w_0 \in \mathcal{V}$,
  \[
  \lim_{k\to+\infty}\| \mathcal{A}_{\eps_k} f_{\eps_k} - R_{\eps_k} w_0 \|_{\eps_k}=0.
  \]

\end{enumerate}

Consider $\mu_\eps^1\geq\mu_\eps^2\geq\dots\geq \mu_\eps^j\geq\dots\to0$ and $\mu_0^1\geq\mu_0^2\geq\dots\mu_0^j\geq \dots\to0$ respectively the sequence of  eigenvalues of $\mathcal{A}_\eps$ and $\mathcal{A}_0$, where each eigenvalue is counted with multiplicity. Let $\{u_\eps^j\}$  be an orthonormal basis of eigenvectors for $\mathcal{A}_\eps$. 

Then the following properties hold:
\begin{enumerate}[(i)]
    \item For every $j\in\mathbb{N}$, 
    \begin{equation}\label{convergenza mu}
        \lim_{\eps\to0^+}\mu_\eps^j=\mu_0^j.
    \end{equation}
    \item There exist $\set{\eps_k}$ a vanishing sequence and $\set{u^j}$ an orthonormal basis of eigenvector of $\mathcal{A}_0$ such that 
    \begin{equation}\label{convergenza estratta autofunzioni}
        \lim_{k\to+\infty}\norma{u_{\eps_k}^j-{R}_{\eps_k} u^j}_{\eps_k}= 0.
    \end{equation}
  \item Assume that $\mu$ is an eigenvalue of multiplicity $l$ for $\mathcal{A}_0$ such that \[\mu=\mu_0^j=\mu_0^{j+1}=\dots=\mu_0^{j+l}.\] Then for all $v$ such that $\mathcal{A}_0 v=\mu v$, there exists a family $\Set{v_\eps}$ such that  $v_\eps\in \spn\{u_\eps^j,\dots ,u_\eps^{j+l}\}$ and
  \begin{equation}\label{conv ad autofunzione}
      \lim_{\eps\to0^+}\norma{ v_\eps-{R}_\eps v}_\eps= 0.
  \end{equation}
  
\end{enumerate}
\end{teor}
For the sake of completeness,  we include the proof of \autoref{convergenza spettri astratto}.
\begin{oss}
Notice that if $\mu=\mu_0^j$ is a simple eigenvalue of $\mathcal{A}_0$, then  \eqref{conv ad autofunzione} implies 
\[\lim_{\eps\to0^+} \norma{u^j_\eps-{R}_\eps v}_\eps=0\]
\end{oss}

We now want to prove that the functionals $\mathcal{T}_{\eps,h}$ and $\mathcal{T}$ satisfy the assumptions of \autoref{convergenza spettri astratto}. 
We can consider the function ${R}_\eps: L^2(\Om)\to L^2(\Om_\eps)$ defined as the extension at $0$ outside of $\Om$. For such a function assumption \Href{H1} is trivially satisfied, taking $\mathcal{V}=L^2(\Om)$.

Moreover, we have the following Lemma.
\begin{lemma}\label{ipotesi conv astratta}
    The family of functionals $\Set{\mathcal{T}_{\eps,h}}$ and $\mathcal{T}_h$ satisfy assumptions \Href{H2}, \Href{H3} and \Href{H4}.
\end{lemma}

\begin{proof}

We start by proving that the family $\Set{\mathcal{T}_{\eps,h}}$ satisfies \Href{H2}. Let $f\in L^2(\Om_\eps)$, by \eqref{eq:epspoinc}, there exists a positive constant $C_p$ such that 
\[\norma{v_{\eps,f}}_{L^2(\Om_\eps)}^2\leq C_p\left(\int_{\Om_\eps}a(d)|\nabla v_{\eps,f}|^2\,dx+\beta \int_{\partial \Om_\eps}v_{\eps,f}^2\,d\Hn\right).\]
Using \eqref{eq forte} and Young's inequality, we obtain that for all $\delta>0$ 
\[\norma{v_{\eps,f}}_{L^2(\Om_\eps)}^2\leq C_p \int_{\Om_\eps}f v_{\eps, f}\, dx\leq  C_p\left(\dfrac{\delta}{2}\norma{f}_{L^2(\Om_\eps)}^2+\dfrac{1}{2\delta}\norma{v_{\eps,f}}_{L^2(\Om_\eps)}^2\right).\]
Choosing $\delta=C_p$, we obtain  
\begin{equation}\label{l2epsapriori}\norma{v_{\eps,f}}_{L^2(\Om_\eps)}\leq {C_p}\norma{f}_{L^2(\Om_\eps)}.\end{equation}
Since $C_p$ does not depend on $\eps$ we deduce \Href{H2}.

Notice that the inequalities above also imply that there exists a constant $C>0$ such that
\begin{equation}\label{boundedness norma X}
    \int_{\Om_\eps}a(d)|\nabla v_{\eps,f}|^2\,dx+\beta \int_{\partial \Om_\eps}v_{\eps,f}^2\,d\Hn\leq C \norma{f}_{L^2(\Om_\eps)}^2.
\end{equation}

We now aim to prove that assumption \Href{H3} is satisfied. Let $f \in L^2(\Omega)$, we extend the functionals $G_{\eps,f}$ and $G_{h,f}$ to the whole space $L^2(\mathbb{R}^n)$  as
\[\tilde{G}_{\varepsilon,f}(w)= \begin{cases}
G_{\eps,f}(w) &\text{ if }w\in H^1(\Omega_\eps),\\[5 pt]

+\infty &\text{ if }w\in L^2(\R^n)\setminus H^1(\Omega_\eps),
\end{cases}\] 
and 
\[\tilde{G}_{f}(w)= \begin{cases}
G_{f}(w) &\text{ if }w\in H^1(\Omega),\\[5 pt]

+\infty &\text{ if }w\in L^2(\R^n)\setminus H^1(\Omega).
\end{cases}\] 
In \cite{DPNST21}, the authors prove that $\tilde{G}_{\varepsilon,f}$ $\Gamma$-converges to $\tilde{G}_{f}$.
By the properties of $\Gamma-$ convergence, we obtain
\[\lim_{\eps\to 0}\norma{v_{\eps,\mathcal{R}_\eps f}-\mathcal{R_\eps} v_{f}}_{L^2(\Om_\eps)}=0.\]

We now want to prove that $\mathcal{T_\eps}$ satisfies \Href{H4}. Let $f_\eps\in L^2(\Om_\eps)$ be such that $\sup_{\eps}\norma{f_\eps}_\eps<+\infty$. Using \eqref{boundedness norma X} and \eqref{l2epsapriori}, we obtain that there exists a constant $C>0$ such that for every $\eps>0$,

\begin{equation}\label{bound norma X}
    \int_{\Om_\eps}a(d)|\nabla v_{\eps,f_\eps}|^2\,dx+\beta \int_{\partial \Om_\eps}v_{\eps,f_\eps}^2\,d\Hn\leq C,
\end{equation}
\[\norma{v_{\eps,f_\eps}}_{L^2(\Om_\eps)}^2\leq C.\]
Hence, $v_{\eps,f_\eps}$ is uniformly bounded in $H^1(\Om)$. As a consequence, there exists $w_0\in L^2(\Om)$ and a vanishing sequence $\set{\eps_k}$ such that 
\[\norma{v_{\eps_k, f_{\eps_k}}-w_0}_{L^2(\Omega)}\to 0.\]
By \eqref{eq:epspoinc2} and \eqref{bound norma X},  there exists a constant $C>0$ such that 
\[\norma{v_{\eps,f_\eps}}_{L^2(\Sigma_\eps)}\leq \eps C\to 0\text{ as }\eps \text{ approaches } 0.\]
\end{proof}

\begin{oss}\label{vepsfeps}
    About assumption \Href{H4}, we remark that, if we assume the family $\set{f_\eps}$ to be extended to zero outside $\Omega_\eps$, up to a subsequence $\set{\eps_k}$, there exists $f\in L^2(\R^n)$ such that $f_{\eps_k}$ converges to $f$ weakly in $L^2(\R^n)$. Then, since for every $w_{\eps_k}$ converging strongly in $L^2(\R^n)$ to a function $w$ we have 
    \[\int_{\Omega_\eps} f_{\eps_k} w_{\eps_k}\,dx\to\int_{\Omega} fw\,dx,\]
    one can adapt the proof of \cite[Theorem 3.1]{DPNST21} to prove that the functional $\tilde{G}_{\eps_k,f_{\eps_k}}$ $\Gamma$-convergence to $\tilde{G}_{f}$. Hence, up to possibly passing to a subsequence, we have that $v_{\eps_k,f_{\eps_k}}$ converges to $v_f$ weakly in $H^1(\Omega)$.
\end{oss}

Since the family $\Set{\mathcal{T}_{\eps,h}}$ satisfies all the assumptions of \autoref{convergenza spettri astratto}, we deduce \autoref{cor3.6}.

\section{Optimisation problems}\label{opt}
Given a fixed amount $m > 0 $ of insulating material, we seek an optimal distribution function $ h $ that minimises the $j$-th eigenvalue $ \lambda^j(h)$ of the limit problem \eqref{eq forte autoval limite}. We recall the definition of the set
\[\mathcal{H}_m(\partial \Om):=\left\{h\in L^1(\partial\Om), \,h\geq 0:\, \int_{\partial\Om}  h\,d\Hn\leq m \right\}.\]
We aim to prove that there exists $\overline{h}\in\mathcal{H}_m$ that maximises $\lambda^j$ over $\mathcal{H}_m$ that saturates the mass constraint.

For each $h\in\mathcal{H}_m$, we define 
\[b_h:=\dfrac{\beta}{1+\beta h},\]
which is the boundary parameter in \eqref{eq forte autoval limite}. Notice that the set
 \[\mathcal{B}_m(\partial\Omega)=\set{b_h:\, h\in \mathcal{H}_m}=\Set{b\in L^\infty(\partial \Om):\, 0<b\leq\beta, \, \int_{\partial \Om} \dfrac{1}{b}\,d\Hn\leq \dfrac{P(\Om)}{\beta} + m},\]
where $P(\Om)$ denotes the perimeter of $\Om$. Given the convexity of the function 
\[x\in(0,+\infty)\mapsto\dfrac{1}{x}\in(0,+\infty),\]
we have that the set $\mathcal{B}_m$ is compact in the the weak-$^*$ topology of $L^\infty(\partial \Om)$.\medskip

\begin{defi}
    We define a topology on $\mathcal{H}_m$ as follows. Let $\set{h_k}$ be a sequence in $\mathcal{H}_m$, We say that $\set{h_k}$ converges to $h\in \mathcal{H}_m $ if and only if the sequence $\set{b_{h_k}}$ converges to $b_h$ in the $L^\infty(\partial\Om)$ weak-$^*$ topology.
\end{defi}
\begin{oss}
    Notice that $\mathcal{H}_m$ with the previously defined topology is compact.
\end{oss}

We aim to prove that all the eigenvalues $\lambda^j(h)$ of the problem \eqref{eq forte autoval limite} are continuous with respect to this topology. 

We now state a uniform Poincaré-type inequality in $\mathcal{H}_m$ that will be useful in the following. The proof is a consequence of the results in the optimisation of $\lambda^1(h)$ in \cite{DPO25}. 
\begin{lemma}
   Fix $m>0$. There exists a constant $C=C(\Om, m)$ such that for all $h\in\mathcal{H}_m$ and for all $w\in H^1(\Om)$
    \begin{equation}\label{poinc unif Hm}
        \int_\Om w^2\,dx\leq C\left[\int_\Om |\nabla w|^2\,dx+\int_{\partial \Om}\dfrac{\beta}{1+\beta h}w^2\,d\Hn\right].
    \end{equation}
    
\end{lemma}
\begin{proof}
    In \cite[Theorem 3.1 and Proposition 3.4]{DPO25} the authors proved the existence of
    \[\min_{h\in \mathcal{H}_m} \lambda^1(h).\]
    Using the variational characterization of $\lambda^1$, we deduce \eqref{poinc unif Hm}, choosing
    \[C^{-1}=\min_{h\in \mathcal{H}_m} \lambda^1(h).\]
    
\end{proof}

We recall that for all $f\in L^2(\Om)$ we denote by $v_{h,f}$ the solution to \eqref{eq forte limite}. Using this notation, we have the following lemma.

\begin{lemma}\label{lemma conv L2 h}
    Let $\{h_k\}\subset \mathcal{H}_m$ such that  $h_k$ converges to $h$ in $\mathcal{H}_m$. Then for all $f\in L^2(\Om)$,
    \[\lim_{k\to +\infty}\norma{v_{h_k,f}-v_{h,f}}_{L^2(\Om)}=0.\]
\end{lemma}
\begin{proof}
    By using the same computations as in \autoref{ipotesi conv astratta}, the Poincaré-like inequality \eqref{poinc unif Hm} implies that there exists $C>0$ such that 
    \[\norma{v_{h_k,f}}_{H^1(\Om)}\leq C\]
    Hence, there exist a subsequence $\set{h_{k_j}}$ and $v\in H^1(\Om)$ such that $v_{h_{k_j},f}$ converges to $v$ weakly in $H^1(\Om)$. Since $v_{h_{k_j},f}$ solves \eqref{eq forte limite}, for all $\varphi\in H^1(\Om)$,
    \begin{equation}\label{eq prima lim}
       \int_\Om \nabla v_{h_{k_j},f}\nabla \varphi\,dx+\int_{\partial \Om}b_{h}v_{h_{k_j},f}\varphi\,d\Hn+\int_{\partial \Om}(b_{h_{k_j}}-b_h)v_{h_{k_j},f}\varphi\,d\Hn=\int_\Om f\varphi\,dx. 
    \end{equation}
    
    Since $v_{h_{k_j},f}\varphi$ converges to $v\varphi$ in $L^1(\partial \Om)$, 
    \[\int_{\partial \Om}(b_{h_{k_j}}-b_h)v_{h_{k_j},f}\varphi\,d\Hn\to 0\text{ as }k \text{ approaches }+\infty.\]
    Taking the limit in \eqref{eq prima lim}, we obtain that for all $\varphi\in H^1(\Om)$
    \[  \int_\Om \nabla v\nabla \varphi\,dx+\int_{\partial \Om}b_{h}v\varphi\,d\Hn=\int_\Om f\varphi\,dx.\]
    Hence $v=v_{h,f}$. Since the limit does not depend on the subsequence, $v_{h_k,f}$ converges to $v_{h,f}$ in $L^2(\Om)$.
\end{proof}
\begin{oss}\label{conv h1}
Notice that the convergence in \autoref{lemma conv L2 h} is actually strong in $H^1(\Omega)$. Indeed, we have shown that any sequence admits a subsequence $\set{v_{h_{k_j}, f}}$ that converges to $v_{h,f}$ weakly in $H^1(\Om)$.  Moreover, using the weak equations, we have that 
\[\int_{\Om}|\nabla v_{h_{k_j},f}|^2\,dx=\int_\Omega f v_{h_{k_j},f}\,dx-\int_{\partial\Omega} b_{h_{k_j}} v_{h_{k_j},f}^2\to\int_\Omega f v_{h,f}\,dx-\int_{\partial\Omega} b_{h} v_{h,f}^2 =\int_{\Om}|\nabla v_{h,f}|^2\,dx. \]
Hence, $v_{h_{k_j},f}$ converges to $v_{h,f}$ strongly in $H^1(\Omega)$.
As the limit does not depend on the subsequence, we deduce that
\[\lim_{k\to+\infty}\norma{v_{h_{k},f}-v_{h,f}}_{H^1(\Om)}=0.\]
\end{oss}

The previous results allows us to prove the existence of a solution to optimisation problems such as maximising the average temperature, that is
\[\max_{h\in\mathcal{H}_m} \int_\Omega v_{h,f}\,dx\]
or minimising the distance from a  desired temperature $v^*$, that is
\[\min_{h\in\mathcal{H}_m} \int_\Omega \abs{v_{h,f}-v^*}^p\,dx,\]
for suitable $p$ (for instance $p\in[1,2^*]$). More generally, we have the following corollary.

\begin{cor}
    Let $j\colon\Omega\times\R\times\R^n\mapsto\R$ and $g\colon\partial\Omega\times\R\times\R^+$ be measurable functions such that \begin{enumerate}[(a)]
        \item  $(y,\xi)\mapsto j(x,y,\xi)$ is lower semicontinuous in $\R\times\R^n$, for almost every $x\in\Omega$,
        \item  $(y,z)\mapsto g(\sigma,y,z)$ is continuous in $\R\times\R^+$, for almost every $\sigma\in\partial\Omega$,
        \item $z\mapsto g(\sigma,y,z)$ is increasing and convex in $\R^+$ for every $(\sigma,y)\in\partial\Omega\times\R$,
        \item $j(x,y,\xi)\ge a_1(x)+a_2(y^2+\abs{\xi}^2)$, where $a_1\in L^1(\Omega)$ and $a_2\in\R$,
        \item $g(\sigma,y,z)\ge a_3(\sigma)+a_4y^2$, where $a_3\in L^1(\partial\Omega)$ and $a_4\in\R$.
    \end{enumerate}
    Let $f\in L^2(\Omega)$, an let
    \[E(h)=\int_\Omega j(x,v_{h,f},\nabla v_{h,f})\,dx+\int_{\partial\Omega} g(\sigma,v_{h,f},h)\,d\Hn.\]
    For every $m>0$ the problem
    \[\min_{h\in\mathcal{H}_m} E(h)\]
    admits a solution.
\end{cor}
\begin{proof}
By classical results and assumptions (a) and (d),  the functional
    \[v\in H^1(\Om)\mapsto\int_\Om j(x,v,\nabla v)\,dx\]
    is lower semicontinuous in the strong $H^1(\Om)$ topology.
    Similarly, since by assumption (c), the function 
    \[w\in\R^+\mapsto g\left(\sigma,y,\dfrac{1}{w}-\dfrac{1}{\beta}\right)\]
    is convex for every $(\sigma,y)\in\partial\Omega\times\R$, we have that, by assumption (b) and (e), the functional
    \[(v,b)\in L^2(\partial\Omega)\times \mathcal{B}_m(\partial\Omega)\mapsto\int_{\partial\Omega} g\left(\sigma,v,\dfrac{1}{b}-\dfrac{1}{\beta}\right)\,d\Hn,\]
    is lower semicontinuous with respect to the strong convergence in $L^2(\partial\Omega)$ and the weak-* convergence in $L^\infty(\partial\Omega)$. Hence, the assertion follows from \autoref{lemma conv L2 h} (and \autoref{conv h1}).
\end{proof}

\begin{teor}\label{cont lambda(h)}
     Let $\{h_k\}\subset \mathcal{H}_m$ such that  $h_k$ converges to $h$ in $\mathcal{H}_m$. Let $\Set{\lambda^j(h_k)}$ and $\Set{\lambda^j(h)}$ denote the sequences of eigenvalues of problems \eqref{eq forte autoval limite}, counted with multiplicity. Let $\{u^j_{h_k}\}$ be an orthonormal basis of eigenfunctions associated with problem \eqref{eq forte autofunz}. Then the following properties hold:
\begin{enumerate}[(i)]
    \item For every $j\in\mathbb{N}$, 
    \begin{equation}\label{convergenza lambda}
        \lim_{k\to +\infty}\lambda^j(h_k)=\lambda^j(h).
    \end{equation}
    \item There exist $\set{h_{k_i}}$ a subsequence of $\set{h_k}$ and $\{u^j(h)\}$ an orthonormal basis of eigenfunctions associated with problem \eqref{eq forte autoval limite} such that 
    \begin{equation}\label{convergenza estratta autofunzioni lambd h}
        \lim_{i\to+\infty}\norma{u^j_{h_{k_i}}- u^j_{h}}_{L^2(\Om)}= 0.
    \end{equation}
  \item Assume that $\lambda$ is an eigenvalue of multiplicity $l$ for problem \eqref{eq forte autoval limite} such that \[\lambda=\lambda^j(h)=\lambda^{j+1}(h)=\dots=\lambda^{j+l}(h).\] 
  Then for all $w$ such that $w$ is an eigenfunction associated to $\lambda$ for \eqref{eq forte autoval limite}, there exists a sequence $\set{w_k}$ such that  $w_k\in \spn\{u^j(h_k),\dots ,u^{j+l}(h_k)\}$ and
  \begin{equation}\label{conv ad autofunzione lambd h}
      \lim_{k\to +\infty}\norma{ w_k- w}_{L^2(\Om)}= 0.
  \end{equation}
  \end{enumerate}
\end{teor}
\begin{proof}

It is sufficient to prove that the sequence of operators $\Set{\mathcal{T}_{h_k}}$ and $\mathcal{T}_h$ satisfy the assumption \Href{H1}--\Href{H4} of \autoref{convergenza spettri astratto}. 

Since the Hilbert space is not varying, \Href{H1} is trivial. Moreover, using the same technique as in \autoref{ipotesi conv astratta}, we can deduce assumptions \Href{H2} and \Href{H4} using the Poincaré-like inequality \eqref{poinc unif Hm}. Finally, the assumption \Href{H3} is satisfied by \autoref{lemma conv L2 h}. 
\end{proof}

\begin{proof}[Proof of \autoref{optlambda}]
We have defined a topology on $\mathcal{H}_m$ for which $\mathcal{H}_m$ is compact and the eigenvalues $\lambda^j$ are continuous, as shown in \autoref{cont lambda(h)}. Therefore, we only need to prove that there exists a minimum for $\lambda^j$ that saturates the mass constraint.

Consider $h_1$ and $h_2$ in $\mathcal{H}_m$ such that $h_1\leq h_2$, then  $b_{h_1}\geq b_{h_2}$. Using the min-max characterization of the eigenvalues, there exists  $\bar{V}$ a $j-$dimensional subspace of $H^1(\Om)$ such that 
\[\lambda^j(h_1)=\max_{v\in \bar{V}\setminus\{0\}}\dfrac{\displaystyle\int_{\Om}|\nabla v|^2\,dx+\int_{\partial\Om} b_{h_1}v^2\,d\Hn}{\displaystyle\int_\Om v^2\,dx}\geq \max_{v\in \bar{V}\setminus\{0\}}\dfrac{\displaystyle\int_{\Om}|\nabla v|^2\,dx+\int_{\partial\Om} b_{h_2}v^2\,d\Hn}{\displaystyle\int_\Om v^2\,dx}\geq \lambda^j(h_2)\]
Hence, we have monotonicity with respect to $h$, and we can deduce the existence of a minimum that saturates the mass constraint.
\end{proof}

\section{Asymptotic development}\label{dev}
In this section, we will prove \autoref{asymptotic}, which provides an asymptotic estimate of the eigenvalues $\lambda_\eps^j(h)$. 
To simplify the notation, we will drop the dependence on $h$. As mentioned in the introduction, we will use the \textit{Lemma on small eigenvalues}, originally introduced in \cite{CV86}. This lemma provides a versatile framework for investigating the asymptotic properties of families of eigenvalues. In particular, we will use the following refinement of the \textit{Lemma}, proven in \cite[Proposition B.1]{ALM22} (see also \cite{C95, O25} for other examples of applications). 
\begin{lemma}(Lemma on small eigenvalues)\label{LSE}
Let $(\mathcal H,(\cdot,\cdot)_{\mathcal H})$ be a real Hilbert space and 
let $\mathcal D$ be a dense subspace of $\mathcal{H}$. 
Let $q:\mathcal D \times \mathcal D \to \mathbb R$ be a symmetric bilinear form bounded from below in the sense that
\[
\inf_{u\in\mathcal D\setminus\{0\}} 
\frac{q(u,u)}{\|u\|_{\mathcal H}^2} > -\infty.
\]

Assume that $q$ has a non-decreasing sequence of eigenvalues 
$\{\eta^k\}$ and that there exists an orthonormal basis 
$\{w_k\}$ of $\mathcal H$ with $w_k\in\mathcal D$ such that
\[
q(w_k,v)=\eta^k (w_k,v)_{\mathcal H}
\qquad \text{for every } k\ge1,\; v\in\mathcal D.
\]

Let $l\in\mathbb{N}_0$ and let $F\subset\mathcal D$ be a subspace of dimension $l+1$. 
Suppose that there exist $i\in\mathbb N$ and $\gamma>0$ such that

\begin{itemize}
\item[\Alabel{1}.] $\eta^k\le -\gamma$ for $k\le i-1$, 
\quad $|\eta^k|\le\gamma$ for $k=i,\dots,i+l$, 
\quad and \quad$\eta^k\ge\gamma$ for $k\ge i+l+1$;

\item[\Alabel{2}.] For every $u\in\mathcal{D}$ and $v\in F$ 
\[\dfrac{\abs{q(u,v)}}{\norma{u}_\mathcal{H} \norma{v}_\mathcal{H}} \le\delta,\qquad \text{with}\qquad \delta<\dfrac{\gamma}{\sqrt{2}}.\]

\end{itemize}

Let $\{\xi^j\}_{j=0,\dots,l}$ be the eigenvalues of $q$ restricted to $F$, 
ordered increasingly, that is,
\[
\xi^j := 
\min_{\substack{G\subset F \\ \dim G=j+1}}
\;\max_{\substack{u\in G \\ u\neq0}}
\frac{q(u,u)}{\|u\|_{\mathcal H}^2}.
\]
Then
\[
|\eta^{i+j}-\xi^j|
\le \frac{4\delta^2}{\gamma}
\qquad \text{for all } j=0,\dots,l.
\]

\end{lemma}

Let $\lambda$ be an eigenvalue of multiplicity $l+1$ for the limit eigenvalue problem \eqref{eq forte autoval limite}, that is
\[\lambda=\lambda^i(h)=\dots=\lambda^{i+l}(h)\]
for some $i\in\N$, and denote by $E(\lambda)$ the associated eigenspace in $H^1(\Omega)$. We want to apply the lemma on small eigenvalues to the quadratic form on $H^1(\Om_\eps)$
\[q_{\eps,\lambda}(v,w)=\int_{\Om_\eps}a(d)\nabla v\nabla w\,dx+\beta\int_{\partial\Om_\eps}vw\,d\Hn-\lambda\int_{\Om_\eps}vw\,dx,\]
to compute
\[\lim_{\eps\to0^+}\dfrac{\lambda_\eps^{i+j}(h)-\lambda^{i+j}(h)}{\eps},\]
for $j=0,\dots,l$. In the following we will denote by $q_\lambda=q_{\eps,\lambda}$.\medskip

We start by noticing that for every $j\in\N$ the $j$-th eigenvalues of $q_\lambda$  is exactly $\eta^j_\eps=\lambda^j_\eps(h)-\lambda^j(h)$. By \autoref{cor3.6}, we have that  $\eta^j_\eps$ converges to zero, as $\eps$ goes to zero, if and only if $j=i,\dots,i+l$, so that assumption \Href{C1} is verified for $\eps$ sufficiently small.\medskip

For every $\eps>0$, and $u\in E(\lambda)$ let 
\[\hat{u}_\eps(x):=\begin{cases}
    u(x) &\text{if }x\in \bar{\Omega},\\[5pt]
    u(\sigma(x))\left(1-\dfrac{\beta d(x)}{\eps(1+\beta h(x))}\right) &\text{if } x\in\Sigma_\eps,\\[5pt]
    0 &\text{if }x \in \R^n\setminus\Om_\eps,
\end{cases}\]
and consider the $l+1$ dimensional subspace of $H^1(\Om_\eps)$ defined as
\[F:=\Set{\mathcal{T}_{\eps,h}\hat{u}_\eps \colon u\in E(\lambda)},\]
where we recall that for a function $f\in L^2(\Om_\eps)$, $\mathcal{T}_{\eps,h}f$ denotes the unique weak solution to the Poisson problem \eqref{eq forte}, that is
\[\int_{\Om_\eps}a(d)\nabla \varphi\nabla \mathcal{T}_{\eps,h}f\,dx+\beta\int_{\partial\Om_\eps}\varphi\mathcal{T}_{\eps,h}f\,d\Hn=\int_{\Om_\eps}f\varphi\,dx,\]
for every $\varphi\in H^1(\Om_\eps)$.\medskip

For every $v\in F$ there exists $u\in E(\lambda)$ such that $v=\mathcal{T}_{\eps,h}(\lambda\hat{u}_\eps)$, and, by \autoref{vepsfeps}, $\mathcal{T}_{\eps,h}(\lambda\hat{u}_\eps)$ converges to $\mathcal{T}_{h}(\lambda u)=u$ so that
\[\lim_{\eps\to0^+}\norma{\mathcal{T}_{\eps,h}(\lambda\hat{u}_\eps) - \hat{u}_\eps}_{L^2(\Omega_\eps)}=0.\]
In order to prove our result, we will need the following estimate. For every $u\in E(\lambda)$ with $\norma{u}_{L^2(\Om_\eps)}=1$
 \begin{equation}\label{5.1strong}
      \int_{\Om_\eps} a(d) \abs{\nabla (\mathcal{T}_{\eps,h}(\lambda\hat{u}_\eps)-\hat{u}_\eps)}^2\,dx+\beta\int_{\partial \Om_\eps} (\mathcal{T}_{\eps,h}(\lambda\hat{u}_\eps)-\hat{u}_\eps)^2\,d\Hn=o(\eps).
   \end{equation}

Such an estimate, by the uniform Poincaré-type inequalities of \autoref{lemma:poinc}, in particular, implies the existence of $\delta=\delta(\eps)$, with $\delta^2=o(\eps)$, such that for every $u\in E(\lambda)$,
\begin{equation}\label{sel1}
    \norma{\mathcal{T}_{\eps,h}(\lambda\hat{u}_\eps) - \hat{u}_\eps}_{L^2(\Omega_\eps)}\le \norma{u}_{L^2(\Om)}\delta. 
\end{equation}

By definition of $F$, we have that for every $v\in F$  and $w\in H^1(\Om_\eps)$

\begin{equation}\label{stima q_lambda v w}
q_\lambda(v,w)=\lambda\int_{\Om_\eps}w\left(\hat{u}_\eps-\mathcal{T}_{\eps,h}(\lambda\hat{u}_\eps)\right)\,dx\le \lambda \norma{w}_{L^2(\Om_\eps)} \norma{\mathcal{T}_{\eps,h}(\lambda\hat{u}_\eps) - \hat{u}_\eps}_{L^2(\Omega_\eps)},
\end{equation}
for some $u\in E(\lambda)$. If \eqref{sel1} holds, we then have 
\[\dfrac{q_\lambda(v,w)}{\norma{v}_{L^2(\Om_\eps)} \norma{w}_{L^2(\Om_\eps)} }\le \lambda  \dfrac{\norma{u}_{L^2(\Om)} }{\norma{\mathcal{T}_{\eps,h}(\lambda\hat{u}_\eps)}_{L^2(\Om_\eps)} }\delta \le C\delta,\]
where we remark that in the last inequality we used the fact that the space $E(\lambda)$ is finite-dimensional. Hence, \eqref{sel1} implies assumption \Href{C2} for $\eps$ sufficiently small. Denoting by $\set{\xi_\eps^j}$ the eigenvalues of the restriction of $q_\lambda$ to $F$ in increasing order, and using the lemma on small eigenvalues, we have that for every $j=0,\dots,l$
\[\abs*{\dfrac{\lambda_\eps^{i+j}(h)-\lambda^{i+j}(h)}{\eps}-\dfrac{\xi_\eps^j}{\eps}}\le \dfrac{C \delta^2}{\eps}=o(1),\]
so that we are left to prove that
\[\lim_{\eps\to0^+}\dfrac{\xi_\eps^j}{\eps}=\zeta_\lambda^{j}(h),\]
 the $(j+1)$-th eigenvalue of the bilinear form 
\[Q_\lambda(u,v) =
\int_{\partial \Om}\dfrac{\beta Hh(2+\beta h)}{2(1+\beta h)^2}uv\,d\Hn -\lambda\int_{\partial\Om}\dfrac{h(3+3\beta h+\beta^2h^2)}{3(1+\beta h)^2}uv\,d\Hn,\]
on the $(l+1)$-dimensional space $E(\lambda)$.\medskip

To prove estimate \eqref{5.1strong}, we adapt the idea in \cite{AC25} (see also \cite{ACNT24}) of "stretching" the function, $v_{\eps,f}=\mathcal{T}_{\eps,h}(f)$, solution to the Poisson problem, to the reference domain $\Sigma_1$ (that is $\Sigma_\eps$ for $\eps=1$), and prove that, in the limit, they are linear in the directions normal to the boundary of $\Omega$. In the particular case of $\mathcal{T}_{\eps,h}(\lambda\hat{u}_\eps)$, we obtain that the "stretched" function converges to the function $\hat{u}_1$. 

Let $d$ be the (signed) distance function from the boundary of $\Omega$ and let 
\[
\Gamma_{t}=\Set{x\in\R^n|\, d(x)<t}\setminus\Omega.
\]
The regularity of $\Omega$ ensures that $\Omega$ satisfies a uniform exterior ball condition, and in particular, there exists $d_0>0$ such that $d\in C^{1,1}(\Gamma_{d_0})$. We define the \emph{stretching diffeomorphism} $\Psi_\eps\in C^{0,1}(\Gamma_{d_0};\Gamma_{\eps d_0})$ as the function defined as
\[
\begin{split}
\Psi_\eps(z)&=\sigma(z)+\eps d(z)\nu_0(z) \\[5 pt] 
&=z+(\eps-1)d(z)\nu_0(z).
\end{split}
\]

 Let $f\in L^2(\Omega)$ and assume $f=0$ in $\R^n\setminus\Omega$, and, in the notation of \autoref{spectr}, let $v_\eps=v_{\eps,f}$ and $v_0=v_f$.  In \cite{AC25}, they prove the following theorem.
\begin{teor}[\cite{AC25}, Theorem 1.1]\label{teorAC}
Fix a positive function $h\in C^{1,1}(\Gamma_{d_0})$ such that $h(x)=h(\sigma(x))$. The function 
    \[
    \tilde{v}_\eps(z)=\begin{cases}
        v_\eps(z) &\text{if }z\in\Omega,\\[5 pt]
        v_\eps(\Psi_\eps(z)) &\text{if }z\in\Sigma_1,
    \end{cases}
    \] 
 is equibounded in $H^1(\Omega_1)$ and, up to a subsequence, converges weakly in $H^1(\Omega_1)$, as $\eps$ goes to $0$, to the function 
    \begin{equation}\label{limfun}\tilde{v}_0(z)=\begin{cases} 
v_0(z) &\text{if }z\in\Omega,\\[5 pt]
    v_0(\sigma(z))\left(1-\dfrac{\beta d(z)}{1+\beta h(z)}\right) &\text{if }z\in\Sigma_1.
\end{cases}\end{equation}
\end{teor}
The proof of \autoref{teorAC} can be divided into three main steps.\medskip

In the first step, one assumes that the family $\set{\tilde{v}_\eps}$ converges weakly in $H^1(\Omega)$ to a function $\tilde{v}$. Then $\tilde{v}$, in $\Sigma_1$, satisfies the equation 
\begin{equation}\label{eqlimweak}
\int_{\Sigma_1} (\nabla\tilde{v}\cdot\nu_0)(\nabla\varphi\cdot\nu_0)J_0(z)\,dz+\beta\int_{\partial\Omega_1} \tilde{v}(\zeta)\varphi(\zeta)\dfrac{J_0(\zeta)}{\sqrt{1+\abs{\nabla h(\zeta)}^2}}\,d\Hn(\zeta)=0, \end{equation}
for every $\varphi\in H^1(\Sigma_1)$ such that $\varphi=0$ on $\partial\Omega$, where 
\[J_0(z)=\prod_{i=1}^{n-1}\dfrac{1}{1+h(z)k_i(\sigma(z))},\]
and $k_i$ are the principal curvatures of $\partial\Omega$. The solution to \eqref{eqlimweak} is uniquely determined by the Dirichlet boundary datum on $\partial\Omega$. Namely, for every $w\in H^1(\Sigma_1)$ the solution to \eqref{eqlimweak} with $\tilde{v}=w$ on $\partial\Omega$ is 
\[\tilde{v}(z)=w(\sigma(z))\left(1-\dfrac{\beta d(z)}{1+\beta h(z)}\right).\]
\medskip

The second step is to prove uniform $H^1$ and $H^2$ density estimates of the type 
\[\int_{\Omega_\eps}a(d)\abs{\nabla v_\eps}^2\,dx+\beta \int_{\partial\Omega_\eps}v_\eps^2\,d\Hn\le C(\Omega,h,f,\beta)\]
\[\int_{\Omega_\eps}a(d)\abs{D^2 v_\eps}^2\,dx+\beta \int_{\partial\Omega_\eps}\abs{\nabla ^{\partial\Omega_\eps}v_\eps}^2\,d\Hn\le C(\Omega,h,f,\beta),\]
where $\nabla^{\partial\Omega_\eps} v := \nabla v - \left(\nabla v\cdot\nu_\eps\right)\nu_\eps$ is the tangential part of the gradient on $\partial\Omega_\eps$. Such estimates imply that 
\begin{equation}\label{eq:entan}\int_{\Sigma_\eps} \abs{\nabla^{\partial\Omega} v_\eps}^2\,dx\le \eps C(\Omega,h,f,\beta).\end{equation}
By direct computation, one can then prove that 
\[
\int_{\Sigma_1}\abs{\nabla^{\partial\Omega}\tilde{v}_\eps}^2\,dz\le \dfrac{C}{\eps}\int_{\Sigma_\eps}\abs{\nabla^{\partial\Omega} v_\eps}^2\,dx,
\]
\[
\int_{\Sigma_1}\abs{\nabla \tilde{v}_\eps\cdot\nu_0}^2\,dz\le\eps C\int_{\Sigma_\eps} \abs{\nabla v_\eps}^2\,dx,    
\]
and
\[
\int_{\partial\Omega_1}\tilde{v}_\eps^2\,d\Hn\le C\int_{\partial\Omega_\eps}v_\eps^2\,d\Hn.
\]
which together with \eqref{eq:entan} and the $H^1$ estimates imply 
\[\int_{\Omega_1}\abs{\nabla \tilde{v}_\eps}^2\,dx+\beta\int_{\partial\Omega_1}\tilde{v}_\eps^2\,dx\le C(\Omega,h,f,\beta),\]
and, hence, the equiboundedness in $H^1(\Omega_1)$.\medskip

The third and final step goes as follows. From the second step, we have that, up to a subsequence, $\set{\tilde{v}_\eps}$ converges weakly in $H^1(\Omega_1)$ to a function $\tilde{v}$. On the other hand, we already know that $v_\eps$ converges to $v_0$ weakly in $H^1(\Om)$, then, by the first step, if $\tilde{v}_\eps$ converges to $\tilde{v}$ weakly in $H^1(\Omega)$, then $\tilde{v}$ is necessarily the $\tilde{v}_0$ defined in \eqref{limfun}. Finally, we notice that the limit $\tilde{v}_0$ does not depend on the subsequence.\medskip

To apply \autoref{teorAC} to the sequence $\mathcal{T}_{\eps,h}(\lambda\hat{u}_\eps)$, however, we need to observe that the assumptions can be slightly weakened. Namely, we have the following theorem.

\begin{teor}\label{teorACeps}
Let $h\in C^{1,1}(\Gamma_{d_0})$ be a positive function such that $h(x)=h(\sigma(x))$. Let $\mathcal{G}=\set{f_\eps}\subset L^2(\Omega_\eps)$ and assume that there exist a constant $C_\mathcal{G}>0$ such that for every $\eps>0$
\begin{equation}\label{fespbound}
 \int_{\Omega_\eps} f_\eps^2\,dx\le C_\mathcal{G},
\end{equation}
and
\begin{equation}\label{fespconc}
   \int_{\Sigma_\eps}f_\eps^2\,dx\le \eps C_\mathcal{G}.
\end{equation}
Let $v_\eps=v_{\eps,f_\eps}$ and 
    \[
    \tilde{v}_\eps(z)=\begin{cases}
        v_\eps(z) &\text{if }z\in\Omega,\\[5 pt]
        v_\eps(\Psi_\eps(z)) &\text{if }z\in\Sigma_1.
    \end{cases}
    \] 
 Then the family $\set{\tilde{v}_\eps}$ is equibounded in $H^1(\Omega_1)$. Moreover, up to a subsequence, $\set {f_\eps}$ converges weakly in $L^2$ to a function $f$ with $f=0$ almost everywhere in $\R^n\setminus\Omega$ and, up to another subsequence, $\set{\tilde{v}_\eps}$ converges weakly in $H^1(\Omega_1)$, as $\eps$ goes to $0$, to the function 
    \begin{equation}\label{limfun1}\tilde{v}_f(z)=\begin{cases} 
v_f(z) &\text{if }z\in\Omega,\\[5 pt]
    v_f(\sigma(z))\left(1-\dfrac{\beta d(z)}{1+\beta h(z)}\right) &\text{if }z\in\Sigma_1.
\end{cases}\end{equation}
\end{teor}

We remark that if $\set{f_\eps}$ is not convergent, then the limit might depend on the subsequence.\medskip

The key ingredient in the proof of \autoref{teorACeps} is establishing an analogue of the energy estimates used in the second step of \autoref{teorAC}, adapted to a sequence of functions. To this end, we require the following theorem, whose proof is deferred to \autoref{energy}.

\begin{restatable}{teor}{COneEnergy}
\label{teor: C11energy}
Let $h\in C^{1,1}(\Gamma_{d_0})$ be a positive function such that $h(x)=h(\sigma(x))$. Let $\mathcal{G}=\set{f_\eps}$ such that \eqref{fespbound} and \eqref{fespconc} hold, then there exists positive constants $\eps_0(\Omega)$, and $C(\Omega,h,C_\mathcal{G},\beta)$ such that if
\[
\eps\norma{h}_{C^{0,1}}\le \eps_0
\]
Let $v_\eps=v_{\eps, f_\eps}$, then
\begin{equation}\label{H1eps}
\int_{\Omega_\eps}a(d)\abs{\nabla v_\eps}^2\,dx+\beta \int_{\partial\Omega_\eps}v_\eps^2\,d\Hn\le C,
\end{equation}
\begin{equation}\label{H2eps}
\int_{\Omega_\eps}a(d)\abs{D^2 v_\eps}^2\,dx+\beta \int_{\partial\Omega_\eps}\abs{\nabla ^{\partial\Omega_\eps}v_\eps}^2\,d\Hn\le C.
\end{equation}
\end{restatable}

\begin{proof}[Proof of \autoref{teorACeps}]
The proof of the theorem follows as the one of \autoref{teorAC}.\medskip

In particular, step one can be proved by the same computations using the non-concentration assumption \eqref{fespconc}. Hence, we have that, if for some sequence $\set{\eps_k}$ going to zero,   $\tilde{v}_{\eps_k}$ converges to a function $\tilde{v}$, then in $\Sigma_1$ 
\[\tilde{v}(z)=\tilde{v}(\sigma(z))\left(1-\dfrac{\beta d(z)}{1+\beta h(z)}\right).\]

The second step follows by the same computations of \autoref{teorAC}, using the uniform energy estimates \eqref{H1eps} and \eqref{H2eps}. 
Hence, by \eqref{fespbound}, the family $\set{f_\eps}$ is bounded in $L^2$ so that  there exist a subsequence $\set{\eps_k}$ and a function $f\in L^2(\R^n)$  such that $f_{\eps_k}$ converges, weakly in $L^2(\R^n)$, to $f$. Moreover, by \eqref{fespconc} $f=0$ almost everywhere in $\R^n\setminus\Omega$. 

To conclude the proof as in step three, we notice that by \autoref{vepsfeps}, up to a subsequence, $\set{v_{\eps_k}}$ converges weakly in $H^1(\Omega)$ to $v_f$, so that by the previous steps, $\tilde{v}_{\eps_k}$ converges, weakly in $H^1(\Omega_1)$ to the function $\tilde{v}_f$ defined in \eqref{limfun1}.
\end{proof}

We notice that \autoref{teorACeps} can also be applied to sequence of eigenfunctions obtaining the following corollary.
\begin{cor}\label{ossnonconc}
Let $\set{u_\eps^j}$ be a family of normalised eigenfunctions of eigenvalue $\lambda_\eps^j$. Then, up to a subsequence, the functions $\tilde{u}_\eps^j=u_\eps^j\circ\Psi_\eps$ converge weakly in $H^1(\Omega_1)$ and strongly in $H^1(\Omega)$ to 
\[\tilde{v}^j(z)=\begin{cases} 
v^j(z) &\text{if }z\in\Omega,\\[5 pt]
    v^j(\sigma(z))\left(1-\dfrac{\beta d(z)}{1+\beta h(z)}\right) &\text{if }z\in\Sigma_1,
    \end{cases}\]
    where $v^j$ is a normalised eigenfunction of eigenvalue $\lambda^j$.
\end{cor}\label{corimp}
\begin{proof}
Recall that, by \eqref{convergenza lambda},  $\lambda_\eps^j$ converges to $\lambda^j$. To prove the assertion, we want to apply  \autoref{teorACeps} to the sequence of eigenfunctions $\set{u_{\eps}^j}$ and notice that the limit in $\Omega$, $v_f$, is indeed a normalised eigenfunction of eigenvalue $\lambda^j$. Indeed, let $f_\eps=\lambda_{\eps}^ju_\eps^j$, then $v_{\eps,f_\eps}=u_\eps^j$. By the convergence of $\lambda_\eps^j$ to $\lambda^j$ and by \eqref{eq:epspoinc2}, we have that the family $\set{f_\eps}$ satisfies assumptions \eqref{fespbound} and \eqref{fespconc}, and we can apply the theorem.\medskip

By \autoref{convergenza spettri astratto}, we have that, up to a subsequence, $\set{u_\eps^j}$ converges, in $L^2$, to an eigenfunction  $v^j$ of associated eigenvalue $\lambda^j$. Moreover, notice that, by the energy estimates \eqref{H1eps} and \eqref{H2eps}, we have that the sequence $\set{u_\eps^j}$ is equibounded in $H^2(\Omega)$, hence, up to a further subsequence, the convergence is strong in $H^1(\Omega)$. Then, up to a subsequence, $\set{f_\eps}$ converges in $L^2$ to $f=\lambda^j v^j$, and $v_f=v^j$. Hence  $\tilde{u}_\eps^j$ converges weakly in $H^1(\Omega)$ to  
\[\tilde{v}_f(z)=\begin{cases} 
v_f(z) &\text{if }z\in\Omega,\\[5 pt]
    v_f(\sigma(z))\left(1-\dfrac{\beta d(z)}{1+\beta h(z)}\right) &\text{if }z\in\Sigma_1.
    \end{cases}\]
\end{proof}\medskip

Finally, to prove \eqref{5.1strong} we will need the following formulas for integrals of $\Sigma_\eps$ and $\partial\Om_\eps$.

\begin{oss}\label{itformulas}
Using the coarea formula with the distance function $d$, we have that for every $g\in L^1(\Omega_\eps)$ 
\[ \int_{\Sigma_\eps}g(x)\,dx=\int_0^{+\infty}\int_{\set{d=t}} g(\xi)\,\chi_{\Sigma_\eps}\!(\xi)\,d\Hn(\xi)\,dt.\]
Let 
\[
\phi_t\colon x\in\Gamma_{d_0}\mapsto x+t\nu_0(x)\in\Gamma_{td_0},
\]
then $\set{d=t}=\phi_t(\partial\Omega)$, and by the area formula on surfaces (see for instance \cite[Theorem 11.6]{M12}) 
\[\begin{split}\int_{\Sigma_\eps}g(x)\,dx&=\int_0^{+\infty}\int_{\partial\Omega} g(\sigma+t\nu_0)\chi_{\Sigma_\eps}(\sigma+t\nu_0) J^\tau\phi_t(\sigma)\,d\Hn(\sigma)\,dt
\end{split}\]
Similarly,
\[
\int_{\partial\Omega_\eps}g(\xi)\,d\Hn(\xi)=\int_{\partial\Omega} g(\sigma+\eps h\nu_0)J^\tau\phi_{\eps h}(\sigma)\sqrt{1+\eps^2\abs{\nabla h}^2}\,d\Hn(\sigma).
\]
Recall that for a perturbation of the identity $\Phi(x)=x+tX$, the tangential Jacobian can be written as
\[J^\tau \Phi=1+t\divv^\tau(X)+R(t,X),\]
where $R$ is a reminder term. 
We have
\begin{equation}\label{eq:intsigma} \int_{\Sigma_\eps}g(x)\,dx=\int_{\partial\Omega}\int_0^{\eps h(\sigma)} g(\sigma+t\nu_0)\left(1+tH(\sigma)+\eps^2 R_1(\sigma,t,\eps)\right)\,dt\,d\Hn\end{equation}
and 
\begin{equation}\label{eq:intdesigma} \int_{\partial\Omega_\eps}g(\sigma)\,d\Hn=\int_{\partial\Omega} g(\sigma+\eps h\nu_0)\left(1+\eps h(\sigma)H(\sigma)+\eps^2 R_2(\sigma,\eps)\right)\,d\Hn,\end{equation}
where the remainder terms $R_1$ and $R_2$ are bounded functions.
\end{oss}

We are finally ready to prove \eqref{5.1strong}.

\begin{lemma} \label{lemma 5.1strong}
    For every $u\in E(\lambda)$ we have that
    \[\lim_{\eps\to0^+}\dfrac{ \displaystyle\int_{\Om_\eps} a(d) \abs{\nabla (\mathcal{T}_{\eps,h}(\lambda\hat{u}_\eps)-\hat{u}_\eps)}^2\,dx+\beta\int_{\partial \Om_\eps} (\mathcal{T}_{\eps,h}(\lambda\hat{u}_\eps)-\hat{u}_\eps)^2\,d\Hn}{\eps}=0.\]
    That is, \eqref{5.1strong} holds.
\end{lemma}
\begin{proof}
Let $v_\eps=\mathcal{T}_{\eps,h}(\lambda\hat{u}_\eps)$ and let $G_\eps=G_{\eps,\lambda \hat{u}_\eps}$, that is
\[G_\eps(\varphi)=\int_{\Om_\eps} a(d) \abs{\nabla \varphi}^2\,dx+\beta\int_{\partial \Om_\eps} \varphi^2\,d\Hn-2\lambda\int_{\Om_\eps}\hat{u}_\eps\varphi.\]
By direct computation, we obtain 
\begin{equation}\label{norma per sigma eps}
\begin{split}
    \int_{\Om_\eps} a(d) \abs{\nabla (v_\eps-\hat{u}_\eps)}^2\,dx+\beta\int_{\partial \Om_\eps} (v_\eps-\hat{u}_\eps)^2\,d\Hn =&G_\eps(\hat{u}_\eps)+ G_\eps(v_\eps)+2\lambda \int_{\Om_\eps} \hat{u}_\eps^2dx+2\lambda \int_{\Om_\eps}v_\eps\hat{u}_\eps\,dx\\
    &- 2\left(\int_{\Om_\eps} a(d) \nabla v_\eps\nabla \hat{u}_\eps\,dx+\beta\int_{\partial \Om_\eps} v_\eps\hat{u}_\eps\,d\Hn\right).
\end{split}
\end{equation}
Since $v_\eps$ solves
\[\int_{\Om_\eps}a(d)\nabla \varphi\nabla v_\eps\,dx+\beta\int_{\partial\Om_\eps}\varphi v_\eps\,d\Hn=\lambda\int_{\Om_\eps}\hat{u}_\eps\varphi\,dx,\]
   we have that 
   \[\int_{\Om_\eps} a(d) \nabla v_\eps\nabla \hat{u}_\eps\,dx+\beta\int_{\partial \Om_\eps} v_\eps\hat{u}_\eps\,d\Hn=\lambda \int_{\Om_\eps} \hat{u}_\eps^2\,dx,\]
   and 
   \[G_\eps(v_\eps)=-\lambda \int_{\Om_\eps}v_\eps\hat{u}_\eps\,dx.\]
   Thus, by \eqref{norma per sigma eps}
   \[\int_{\Om_\eps} a(d) \abs{\nabla (v_\eps-\hat{u}_\eps)}^2\,dx+\beta\int_{\partial \Om_\eps} (v_\eps-\hat{u}_\eps)^2\,d\Hn=G_\eps(\hat{u}_\eps)-G_\eps(v_\eps).\]
   To prove \eqref{5.1strong}, we prove that 
   \[\lim_{\eps\to0^+}\dfrac{G_\eps(\hat{u}_\eps)-G_\eps(v_\eps)}{\eps}=0.\]\medskip
   
   Let 
   \[G_0(\varphi)=\int_\Om \abs{\nabla\varphi}^2\,dx+\beta\int_{\partial\Om}\dfrac{\varphi^2}{1+\beta h}\,d\Hn-2\lambda\int_\Om u\varphi\,dx,\]
   which admits $u$ as the unique minimiser. Then
   \[\dfrac{G_\eps(\hat{u}_\eps)-G_\eps(v_\eps)}{\eps}\le\dfrac{G_\eps(\hat{u}_\eps)-G_0(u)-(G_\eps(v_\eps)-G_0(v_\eps))}{\eps}.\]
   We start by estimating 
   \[G_\eps(\hat{u}_\eps)-G_0(u)=\eps\int_{\Sigma_\eps}\abs{\nabla \hat{u}_\eps}^2\,dx+\beta\int_{\partial\Om_\eps}\hat{u}^2_\eps\,d\Hn-\beta\int_{\partial\Om}\dfrac{u^2}{1+\beta h}\,d\Hn-2\lambda\int_{\Sigma_\eps}\hat{u}^2_\eps\,dx,\]
   from above. For every $x\in\Sigma_\eps$, $x=\sigma+t\nu_0(\sigma)$ for some $\sigma\in\partial\Omega$ and $t\in(0,h(\sigma))$, and, by direct computations 
   \[\abs{\nabla \hat{u}_\eps}^2(x)\le\dfrac{\beta^2 u^2(\sigma)}{\eps^2(1+\beta h)^2}+C\left(\abs{\nabla u(\sigma)}^2+u^2(\sigma)\right),\]
   where $C=C(h,\beta)$ and
   \[\hat{u}_\eps(x)=u(\sigma)\left(1-\dfrac{\beta t}{\eps(1+\beta h)}\right).\]
   using \eqref{eq:intsigma}, we have
   \begin{equation}\label{suphat1}\begin{split}\eps\int_{\Sigma_\eps}\abs{\nabla\hat{u}_\eps^2}\,dx\le&\beta^2\int_{\partial\Om} \dfrac{u^2}{(1+\beta h)^2}\int_0^{h(\sigma)}\left(1+\eps t H\right)\,dt\,d\Hn+\eps^2 C(\Omega,h,\beta)\int_{\partial\Omega}\left(\abs{\nabla u}^2+u^2\right)\,d\Hn\\[10pt]
   \le&\beta^2\int_{\partial\Om} \dfrac{u^2h}{(1+\beta h)^2}\left(1+\dfrac{\eps h H}{2}\right)\,d\Hn+\eps^2C(\Omega,h,\lambda,\beta), \end{split}\end{equation}
   and
   \begin{equation}\label{suphat2}\begin{split}\int_{\Sigma_\eps}\hat{u}_\eps^2\,dx\ge&\eps\int_{\partial\Omega}u^2\int_0^{h(\sigma)}\left(1-\dfrac{\beta t}{1+\beta h}\right)^2\,dt\,d\Hn-\eps^2C(\Omega,h,\beta)\int_{\partial\Om}u^2\,\Hn\\[10pt]
    \ge&\eps\int_{\partial\Omega}\dfrac{h(3+3\beta h+\beta^2 h^2)}{3(1+\beta h)^2}u^2\,d\Hn-\eps^2C(\Omega,h,\lambda,\beta).\end{split}\end{equation}\medskip

On the other hand, using \eqref{eq:intdesigma}, we have
\begin{equation}\label{suphat3}\beta\int_{\partial\Omega_\eps}\hat{u}_\eps^2\,d\Hn\le\beta\int_{\partial\Om}\dfrac{u^2}{(1+\beta h)^2}(1+\eps h H)\,d\Hn+\eps^2 C(\Omega,h,\lambda,\beta).\end{equation}
Putting together estimates \eqref{suphat1}, \eqref{suphat2} and \eqref{suphat3}, we get
\begin{equation}\label{suphatlast}G_\eps(\hat{u}_\eps)-G_0(u)\le\eps\left[\beta\int_{\partial\Om}\dfrac{h H(2+\beta h)}{2(1+\beta h)^2}u^2\,d\Hn-2\lambda\int_{\partial\Om}\dfrac{h(3+3\beta h+\beta^2 h^2)}{3(1+\beta h)^2}u^2\,d\Hn\right]+\eps^2 C(\Omega,h,\lambda,\beta) \end{equation}

   We now want to estimate from below 
   \[G_\eps(v_\eps)-G_0(v_\eps)=\eps \int_{\Sigma_\eps}|\nabla v_\eps|^2\,dx+\beta\int_{\partial \Om_\eps}v_\eps^2\,d\Hn-\beta \int_{\partial \Om}\dfrac{v_\eps^2}{1+\beta h}\,d\Hn-2\lambda \int_{\Sigma_\eps}\hat u_\eps v_\eps\,dx.\]

Using \eqref{eq:intsigma} and \eqref{eq:intdesigma}, and letting $\bar{R}>0$ such that $\abs{R_1},\abs{R_2}\le \bar{R}$, we have
\begin{equation}\label{stima 1 aeps}
   \eps \int_{\Sigma_{\eps}} |\nabla v_\eps|^2 \, dx \geq \eps\int_{\partial \Omega} \int_0^{{\eps} h(\sigma)} \abs{\nabla v_\eps(\sigma + t \nu_0)}^2 \left(1 + t H(\sigma) - {\eps}^2 \bar{R}\right) \, dt \, d\mathcal{H}^{n-1}
\end{equation}

and
\begin{equation}\label{stima 2 aeps}
    \beta \int_{\partial \Omega_{\eps}} v_\eps^2 \,d\Hn \geq \beta \int_{\partial \Omega} v_\eps^2(\sigma + {\eps} h(\sigma)\nu_0(\sigma)) \left(1 + {\eps} h(\sigma) H(\sigma) - {\eps}^2 \bar{R}\right) \, d\mathcal{H}^{n-1}.
\end{equation}

For ${\eps}$ sufficiently small,   for every $\sigma \in \partial \Omega$ and $0 < t < {\eps} h(\sigma)$, we have that $1 + t H(\sigma) > 0$. Using Hölder’s inequality, we obtain
\[\begin{split}
\int_0^{{\eps} h} |\nabla v_\eps(\sigma + t \nu_0)|^2 (1 + t H) \, dt &\geq \dfrac{1}{{\eps} h} \left( \int_0^{{\eps} h} |\nabla v_\eps(\sigma + t \nu_0)| \sqrt{1 + t H} \, dt \right)^2\\[15 pt]
&\geq \dfrac{1}{{\eps} h} \left( \int_0^{{\eps} h} \dfrac{d}{dt}(v_\eps(\sigma + t \nu_0)) \sqrt{1 + t H} \, dt \right)^2\\[15 pt]
&= \dfrac{1}{{\eps} h} \left( v_\eps(\sigma + {\eps} h \nu_0)\sqrt{1 + {\eps} h H} - \left(v_\eps(\sigma) + \int_0^{{\eps} h} \dfrac{H v_\eps(\sigma + t \nu_0)}{2 \sqrt{1 + t H}} \, dt \right) \right)^2.
\end{split}\]

By Young’s inequality, we deduce that for every $\alpha > 0$  
\begin{equation}\label{stima 3 aeps}
    \begin{split}
        \int_0^{{\eps} h(\sigma)} |\nabla v_\eps(\sigma + t \nu_0)|^2 (1 + t H) \, dt 
\geq& \dfrac{(1 - \alpha)(1 + {\eps} h H) v_\eps(\sigma + {\eps} h \nu_0)^2}{{\eps} h}\\[15 pt]
&+ \dfrac{1}{{\eps} h} \left(1 - \dfrac{1}{\alpha} \right) \left(v_\eps(\sigma) + \int_0^{{\eps} h(\sigma)} \dfrac{H v_\eps(\sigma + t \nu_0)}{2 \sqrt{1 + t H}} \, dt \right)^2.
    \end{split}
\end{equation}

Joining \eqref{stima 1 aeps}, \eqref{stima 2 aeps} and \eqref{stima 3 aeps}, we obtain 

\[\begin{split} G_\eps(v_\eps)-G_0(u)
\geq& \int_{\partial \Omega} \dfrac{1}{ h(\sigma)} (1 - \alpha + \beta h)(1 + {\eps} h H) v_\eps^2(\sigma + {\eps} h \nu_0) \, d\Hn\\[15 pt]
&+ \int_{\partial \Omega} \dfrac{1}{ h} \left( \left(1 - \dfrac{1}{\alpha} \right) \left(v_\eps(\sigma) + \int_0^{{\eps} h} \dfrac{H v_\eps(\sigma + t \nu_0)}{2 \sqrt{1 + t H}} \, dt \right)^2 
- \dfrac{\beta h v_\eps^2(\sigma)}{1 + \beta h} \right) d\mathcal{H}^{n-1} \\
&- \eps^2 \bar{R} \mathcal{R}({\eps}, v_\eps)-2\lambda \int_{\Sigma_\eps} v_\eps \hat{u}_\eps\,dx
\end{split}\]

where
\[
\mathcal{R}({\eps}, v_\eps) = {\eps} \int_{\partial \Omega} \int_0^{{\eps} h(\sigma)} |\nabla v_\eps(\sigma + t \nu_0)|^2 \, d\mathcal{H}^{n-1} \, dt 
+ \beta \int_{\partial \Omega} v_\eps(\sigma + {\eps} h\nu_0)^2 \, d\mathcal{H}^{n-1}\le C(\Omega,h,\lambda,\beta).\]
Then, choosing \(\alpha = \alpha(\sigma) = 1 + \beta h(\sigma)\),

\begin{equation}\label{inf2tolast}G(v_\eps)-G_0(v_\eps)\geq \int_{\partial \Omega} \dfrac{\beta H v_\eps(\sigma)}{ (1 + \beta h)} \int_0^{{\eps} h(\sigma)} \dfrac{v_\eps(\sigma + t \nu_0)}{\sqrt{1 + t H}} dt\, d\mathcal{H}^{n-1}-2\lambda \int_{\Sigma_\eps} v_\eps \hat{u}_\eps\,dx-\eps^2 C(\Omega,h,\lambda,\beta).
\end{equation}

Using \eqref{eq:intsigma}, we obtain that    
\begin{equation}\label{infsigma}\begin{split}
\int_{\Sigma_{\eps}} \hat{u}_\eps v_\eps\, dx =& \int_{\partial \Omega} \int_0^{{\eps} h(\sigma)} \hat{u}_\eps v_\eps(\sigma + t \nu) J(t,\sigma) \, dt  d\Hn\\[10pt]
=& \eps \int_{\partial \Omega} u(\sigma) \int_0^{h(\sigma)} \left(1-\dfrac{\beta s}{1+\beta h}\right)\tilde{v}_\eps(\sigma + s\nu) J(\eps s, \sigma) \, ds \, d\Hn.
\end{split}\end{equation}
If $\eps$ is sufficiently small we have that $\norma{(1+\eps s H)^{-1/2}-1}_\infty,\norma{J(\eps s, \sigma)-1}_\infty\le \eps C(\Om,h)$ so that, putting together \eqref{inf2tolast} and \eqref{infsigma}, we have

\begin{equation}\label{inflast}
\begin{split}G(v_\eps)-G_0(v_\eps)\geq& \eps\beta\int_{\partial \Omega} \dfrac{ H v_\eps(\sigma)}{ (1 + \beta h)} \int_0^{h(\sigma)} \tilde{v}_\eps(\sigma + s \nu_0)ds\, d\mathcal{H}^{n-1}\\[10pt]&-2\lambda\eps  \int_{\partial \Omega} u(\sigma) \int_0^{h(\sigma)} \left(1-\dfrac{\beta s}{1+\beta h}\right)\tilde{v}_\eps(\sigma + s\nu) \, ds \, d\Hn-\eps^2 C(\Omega,h,\lambda,\beta).
\end{split}
\end{equation}

Joining \eqref{suphatlast} and \eqref{inflast} we have
\[\begin{split}\dfrac{G_\eps(\hat{u}_\eps)-G_\eps(v_\eps)}{\eps}\le& \beta\int_{\partial\Omega}H\left( \dfrac{h (2+\beta h)}{2(1+\beta h)^2}u^2- \dfrac{ v_\eps}{ (1 + \beta h)} \int_0^{h} \tilde{v}_\eps(\sigma + s \nu_0)ds\right)\, d\mathcal{H}^{n-1}\\[10pt]
&-2\lambda\int_{\partial\Om}u\left(\dfrac{h(3+3\beta h+\beta^2h^2)}{3(1+\beta h)^2}u - \int_0^{h} \left(1-\dfrac{\beta s}{1+\beta h}\right)\tilde{v}_\eps(\sigma + s\nu) \, ds \right)\, d\Hn\\[10pt]
&+\eps C(\Omega,h,\lambda,\beta).
\end{split}\]

By \autoref{teorACeps} $\tilde{v}_\eps$ converges weakly $H^1(\Omega)$ to 
\[\tilde{u}(z)=\begin{cases}
    u(z) &\text{if }z\in\Omega,\\[5pt]
    u(\sigma(z))\left(1-\dfrac{\beta d(z)}{1+\beta h(z)}\right) &\text{if }z\in\Sigma_1.
\end{cases}\]
In particular,  we have the following convergences in $L^2(\partial\Omega)$
\[\begin{split}v_\eps&\longrightarrow u,\\[10pt]   
\int_0^{h(\sigma)}\tilde{v}_\eps(\sigma+s\nu_0)\,ds&\longrightarrow \int_0^{h(\sigma)}\tilde{u}(\sigma+s\nu_0)\,ds,\\[10pt]
\int_0^{h(\sigma)}\left(1-\dfrac{\beta s}{1+\beta h}\right)\tilde{v}_\eps(\sigma+s\nu_0)\,ds&\longrightarrow \int_0^{h(\sigma)}\left(1-\dfrac{\beta s}{1+\beta h}\right)\tilde{u}(\sigma+s\nu_0)\,ds.
\end{split}\]
Then
\[\begin{split}\lim_{\eps\to0^+}\int_{\partial \Omega} \dfrac{ H v_\eps(\sigma)}{ (1 + \beta h)} \int_0^{h(\sigma)} \tilde{v}_\eps(\sigma + s \nu_0)ds\, d\mathcal{H}^{n-1}=&\int_{\partial \Omega} \dfrac{ H u(\sigma)}{ (1 + \beta h)} \int_0^{h(\sigma)} \tilde{u}(\sigma + s \nu_0)ds\, d\mathcal{H}^{n-1}\\[10pt]&=\int_{\partial\Om}\dfrac{h H(2+\beta h)}{2(1+\beta h)^2}u^2\,d\Hn,\end{split}\]
\[\begin{split}\lim_{\eps\to0^+}\int_{\partial \Omega} u(\sigma) \int_0^{h(\sigma)} \left(1-\dfrac{\beta s}{1+\beta h}\right)\tilde{v}_\eps(\sigma + s\nu) \, ds \, d\Hn=&\int_{\partial \Omega} u(\sigma) \int_0^{h(\sigma)} \left(1-\dfrac{\beta s}{1+\beta h}\right)\tilde{u}(\sigma + s\nu) \, ds \, d\Hn\\[10pt]&=\int_{\partial\Om}\dfrac{h(3+3\beta h+\beta^2 h^2)}{3(1+\beta h)^2}u^2\,d\Hn\end{split}\]
and
\[\lim_{\eps\to0^+}\dfrac{G_\eps(\hat{u}_\eps)-G_\eps(v_\eps)}{\eps}=0,\]
which implies \eqref{5.1strong} and the assertion.

\end{proof}

\begin{oss}
Let $v_\eps=\mathcal{T}_{\eps,h}(\lambda\hat{u}_\eps)$. It actually holds true that 
\[\norma{v_\eps-\hat{u}_\eps}_{L^2(\Om_\eps)}\le \eps C. \]
Indeed, on $\Sigma_\eps$, by \eqref{eq:epspoinc2}, we have
   \[\norma {v_\eps-\hat{u}_\eps}_{L^2(\Sigma_\eps)}^2\leq \eps C_p\left[ \int_{\Om_\eps} a(d) \abs{\nabla (v_\eps-\hat{u}_\eps)}^2\,dx+\beta\int_{\partial \Om_\eps} (v_\eps-\hat{u}_\eps)^2d\Hn\right],\]
   hence by \eqref{5.1strong}, in the proof of the previous lemma, we have 
   \[\norma {v_\eps-\hat{u}_\eps}_{L^2(\Sigma_\eps)}^2\leq \eps^2 C.\]
To prove the estimate in $\Omega$, we use the fact that the function $w_\eps=v_\eps-\hat{u}_\eps$ is a solution to 
 \[\int_\Om \nabla w_\eps\nabla\varphi\,dx+\beta\int_{\partial\Om}\dfrac{ w_\eps\varphi}{1+\beta h}\,d\Hn=\int_{\partial\Om}g_\eps\varphi\,d\Hn,\]
for every $\varphi\in H^1(\Omega)$, where the function  \[g_\eps=\dfrac{\partial v_\eps^- }{\partial \nu_0}+\dfrac{\beta v_\eps}{1+\beta h},\]
is well defined and in $L^2(\partial\Om)$ 
as $v_\eps\in H^1(\Om_\eps)\cap H^2(\Omega\cup\Sigma_\eps)$. We recall that $v_\eps^-$ and $v_\eps^+$ denote the traces of $v_\eps$ on $\partial\Omega$ from $\Omega$ and from $\Sigma_\eps$ respectively. By H\"older's inequality, there exists a constant $C=C(\Om,h)$ such that 
    \[\int_\Om \abs{\nabla w_\eps}^2\,dx+\beta\int_{\partial\Om}\dfrac{ w_\eps^2}{1+\beta h}\,d\Hn\le C\int_{\partial\Om}g_\eps^2\,d\Hn,\]
    and, in particular, by the Poincarè-type inequality \eqref{poinc unif Hm}
    \begin{equation}\label{bestgeps}\int_\Om w_\eps^2\,dx\le C\int_{\partial\Om}g_\eps^2\,d\Hn.\end{equation}
    In order to estimate the $L^2$ norm of $g_\eps$ on $\partial\Om$ we argue as in \cite[Theorem 2.1]{BCF80}. On $\partial\Om,$ $v_\eps$ must satisfy the transmission condition
    \[\dfrac{\partial v_\eps^- }{\partial \nu_0}=\eps\dfrac{\partial v_\eps^+ }{\partial \nu_0}.\]
    Hence, for $\Hn$-almost every $\sigma\in \partial\Om$ 
    \[g_\eps(\sigma)=\eps\dfrac{\partial v_\eps^+ (\sigma)}{\partial \nu_0}+\dfrac{\beta v_\eps(\sigma)}{1+\beta h}.\]
    For $\Hn$-almost every $\sigma\in \partial\Om$ and $t\in(0,\eps h(\sigma))$ we have
     \[\dfrac{\partial v_\eps (\sigma+t\nu_0)}{\partial \nu_0}=-\int_t^{\eps h(\sigma)}\dfrac{\partial^2 v_\eps (\sigma+s\nu_0)}{\partial \nu_0^2} \,ds+\dfrac{\partial v_\eps (\sigma+\eps h\nu_0)}{\partial \nu_0}\]
     and
    \[\begin{split}v_\eps(\sigma)&=-\int_0^{\eps h(\sigma)}\dfrac{\partial v_\eps (\sigma+t\nu_0)}{\partial \nu_0} \,dt+v_\eps(\sigma+\eps h\nu_0)\\[10pt]
    &=\int_0^{\eps h(\sigma)}\int_t^{\eps h(\sigma)}\dfrac{\partial^2 v_\eps (\sigma+s\nu_0)}{\partial \nu_0^2} \,ds \,dt-\eps h \dfrac{\partial v_\eps (\sigma+\eps h\nu_0)}{\partial \nu_0}+v_\eps(\sigma+\eps h\nu_0).
    \end{split}\]
    Then, using the Robin boundary condition on $\partial\Om_\eps$
    \begin{equation}\label{Rbcomeps}\eps \dfrac{\partial v_\eps (\sigma+\eps h\nu_0)}{\partial \nu_\eps}+\beta v_\eps(\sigma+\eps h\nu_0)=0,\end{equation}
    we have
    \[\begin{split}g_\eps(\sigma)=& -\eps \int_0^{\eps h(\sigma)}\dfrac{\partial^2 v_\eps (\sigma+s\nu_0)}{\partial \nu_0^2} \,ds+\dfrac{\beta}{1+\beta h}\int_0^{\eps h(\sigma)}\int_t^{\eps h(\sigma)}\dfrac{\partial^2 v_\eps (\sigma+s\nu_0)}{\partial \nu_0^2} \,ds \,dt\\[10pt]
    &+\dfrac{\eps}{1+\beta h}\nabla v_\eps(\sigma+\eps h\nu_0)\cdot(\nu_0(\sigma)-\nu_\eps(\sigma+\eps h\nu_0)).
    \end{split}\]
so that 
\[g_\eps^2(\sigma)\le \eps^2C(h,\beta)\left[\left( \int_0^{\eps h}\dfrac{\partial^2 v_\eps (\sigma+s\nu_0)}{\partial \nu_0^2} \,ds\right)^2+v_\eps(\sigma+\eps h\nu_0)^2+\eps^2 \abs*{\nabla^{\partial\Omega_\eps}v_\eps(\sigma+\eps h\nu_0)}^2\right],\]
where we used the fact that $\abs{\nu_0-\nu_\eps}\le C(\norma{\nabla h}_\infty)\eps$ and decomposed the gradient of $v_\eps$ in its normal and tangent component to $\partial\Omega_\eps$ and used the Boundary condition \eqref{Rbcomeps} to estimate the normal component. Finally, integrating on $\partial\Omega$ and using \eqref{eq:intsigma} and \eqref{eq:intdesigma}, we have 
\[\int_{\partial\Om}g_\eps^2\,d\Hn\le \eps^2 C\left[\eps\int_{\Sigma_\eps}\abs{D^2 v_\eps}^2\,dx+\eps^2\int_{\partial\Om_\eps}\abs{\nabla^{\partial\Om_\eps} v_\eps}^2\,d\Hn+\int_{\partial\Om_\eps}v_\eps^2\,d\Hn\right],\]
where $C=C(\Omega,h,\beta)$ and the expression in the square brackets, using the energy estimates proved in \autoref{teor: C11energy}, is bounded by a constant depending on $\Omega, h, \beta$ and $\lambda$. Finally, by \eqref{bestgeps}, we have
\[\norma{v_\eps-\hat{u}_\eps}_{L^2(\Om)}\le \eps C,\]
hence the assertion.

\end{oss}

\begin{prop}\label{xitozeta}
    For every $u_1,u_2\in E(\lambda)$ we have
    \[q_\lambda(\mathcal{T}(\lambda\hat{u}_{1,\eps}),\mathcal{T}(\lambda\hat{u}_{2,\eps}))=\eps Q_\lambda(u_1,u_2)+o(\eps).\]
    In particular, by the continuity of matrices' eigenvalues, we have that the eigenvalues of $\eps^{-1}q_\lambda$ restricted to $F$ converge to the ones of $Q_\lambda$ restricted to $E(\lambda)$.
\end{prop}
\begin{proof}
    Let $v_{i,\eps}=\mathcal{T}(\lambda\hat{u}_{i,\eps})$ and $w_{i\eps}=v_{i,\eps}-\hat{u}_{i,\eps}$ for $i=1,2$. We start by noticing that 
    \begin{equation}\label{appross qlamba 1}q_\lambda(v_{1,\eps},v_{2,\eps})=q_\lambda(\hat{u}_{1,\eps},\hat{u}_{2,\eps})+o(\eps),
    \end{equation}
    indeed, by direct computations, we have
    \[q_\lambda(v_{1,\eps},v_{2,\eps})-q_\lambda(\hat{u}_{1,\eps},\hat{u}_{2,\eps})=q_\lambda(v_{1,\eps},w_{2,\eps})+q_\lambda(v_{2,\eps},w_{1,\eps})-q_\lambda(w_{1,\eps},w_{2,\eps}).\]
    Since $q_\lambda$ is continuos in $H^1(\Om_\eps)$ (uniformly in $\eps$) with respect to the weighted norm
    \[\left(\int_{\Om_\eps}a(d)\abs{\nabla w}^2\,dx+\beta\int_{\partial\Om_\eps} w^2\,d\Hn\right)^{1/2}\]
    using \eqref{5.1strong}, we have that 
    \[q_\lambda(w_{1,\eps},w_{2,\eps})=o(\eps),\]
    while, for $i,j=1,2$, by \eqref{stima q_lambda v w}
    \[q(v_{i,\eps},w_{j,\eps})=-\lambda \int_{\Om_\eps} w_{i,\eps}w_{j,\eps}\,dx=o(\eps).\]

    We now need to estimate $q_\lambda(\hat{u}_{1,\eps},\hat{u}_{2,\eps})$. Using the fact that $u_1,u_2\in E(\lambda)$ we have 
    \[q_\lambda(\hat{u}_{1,\eps},\hat{u}_{2,\eps})=\eps\int_{\Sigma_\eps}\nabla \hat{u}_{1,\eps}\nabla\hat{u}_{2,\eps}\,dx+\beta\int_{\partial\Om}\hat{u}_{1,\eps}\hat{u}_{2,\eps}\,d\Hn-\beta\int_{\partial\Om} \dfrac{u_1 u_2}{1+\beta h}\,d\Hn-\lambda\int_{\Sigma_\eps}\hat{u}_{1,\eps}\hat{u}_{2,\eps}\,dx.\]
For every $x\in\Sigma_\eps$, $x=\sigma+t\nu_0(\sigma)$ for some $\sigma\in\partial\Omega$ and $t\in(0,h(\sigma))$, and, by direct computations 
\[
\nabla \hat{u}_{1,\eps}(x)\cdot \nabla\hat{u}_{2,\eps}(x) = \dfrac{\beta^2 u_1(\sigma)u_2(\sigma)}{{\eps}^2 (1 + \beta h)^2} + M_\eps(x),
\]

where 
\[\abs{M_\eps(x)}\le C(h,\beta)(\abs{\nabla u_1(\sigma)}^2+\abs{\nabla u_2(\sigma)}^2+u_1(\sigma)^2+u_2(\sigma)^2.\]Then arguing as in \eqref{suphat1}, \eqref{suphat2} and \eqref{suphat3}, we have that
\begin{equation}\label{suphat1bis}
\eps\int_{\Sigma_\eps}\nabla\hat{u}_{1,\eps}\nabla\hat{u}_{2,\eps}\,dx=\beta^2\int_{\partial\Om} \dfrac{u_1u_2h}{(1+\beta h)^2}\left(1+\dfrac{\eps h H}{2}\right)\,d\Hn+\eps^2\mathcal{R}_1(\eps,u_2,u_2),\end{equation}
\begin{equation}\label{suphat2bis}
\int_{\Sigma_\eps}\hat{u}_{1,\eps}\hat{u}_{2,\eps}\,dx=\eps\int_{\partial\Omega}\dfrac{h(3+3\beta h+\beta^2 h^2)}{3(1+\beta h)^2}u_1u_2\,d\Hn+\eps^2\mathcal{R}_2(\eps, u_1,u_2),\end{equation}
\begin{equation}\label{suphat3bis}
\beta\int_{\partial\Omega_\eps}\hat{u}_{1,\eps}\hat{u}_{2,\eps}\,d\Hn=\beta\int_{\partial\Om}\dfrac{u_1u_2}{(1+\beta h)^2}(1+\eps h H)\,d\Hn+\eps^2 \mathcal{R}_3(\eps,u_1,u_2).\end{equation}
where $\abs{\mathcal{R}_1},\abs{\mathcal{R}_2},\abs{\mathcal{R}_3}\le C(\Omega,h,\lambda,\beta)$. Then, putting together  \eqref{suphat1bis}, \eqref{suphat2bis} and \eqref{suphat3bis} we get 
\[q_\lambda(\hat{u}_{1,\eps},\hat{u}_{2,\eps})=\eps Q_\lambda(u_1,u_2)+O(\eps^2),\]
which, together with \eqref{appross qlamba 1} gives the assertion.
\end{proof}

 We are now ready to prove \autoref{asymptotic}.
 \begin{proof}[Proof of \autoref{asymptotic}]

By \eqref{5.1strong} and the Lemma on small eigenvalues (\autoref{LSE}) we have that for every $j=0,\dots,l$
\[\abs*{\dfrac{\lambda_\eps^{i+j}(h)-\lambda^{i+j}(h)}{\eps}-\dfrac{\xi_\eps^j}{\eps}}\le \dfrac{C \delta^2}{\eps}=o(1). \]
Where we recall that $\set{\xi_\eps^j}$ denotes the eigenvalues of the restriction of $q_\lambda$ to $F$. Then, by \autoref{xitozeta} we have
\[\lim_{\eps\to0^+}\dfrac{\xi_\eps^j}{\eps}=\zeta_\lambda^{j}(h),\]
hence
\[\lim_{\eps\to0^+}\dfrac{\lambda_\eps^{i+j}(h-\lambda^{i+j}(h)}{\eps})=\lim_{\eps\to0+}\dfrac{\xi_\eps^j}{\eps}=\zeta_\lambda^{j}(h).\] 
 \end{proof}

\section{Remarks on the Dirichlet case}\label{Dirichlet}
In this section we remark on how (and when) the presented result generalise to the case in which we replace the Robin boundary condition, in the original problem, with a Dirichlet one.\medskip

Fix $h$ a strictly positive Lipschitz function on $\partial\Omega$. For every $\eps>0$ we consider the following eigenvalue problem with Dirichlet boundary condition.
\begin{equation}\label{eqD forte autofunz}
\begin{cases}
-\Delta u = \lambda u & \text{in } \Omega, \\[5 pt]
-\eps\Delta u = \lambda u  & \text{in } \Sigma_\eps, \\[5 pt]
u^-=u^+ & \text{on } \partial \Om, \\[5 pt]
\displaystyle\dfrac{\partial {{u}}^-}{\partial \nu_0} = \varepsilon \dfrac{\partial {u}^+}{\partial \nu_0} & \text{on } \partial \Omega,\\[10 pt]

u = 0 & \text{on } \partial \Omega_\eps, 
\end{cases}
\end{equation}
which admits a discrete spectrum 
\[0<\lambda^1_{ \eps}(h)\le\lambda^2_{ \eps}(h)\le\dots\le\lambda^j_{ \eps}(h)\le\dots\to+\infty.\]
In such case the limit eigenvalue problem is the following Robin boundary value problem on $\Omega$
\begin{equation}\label{eqD forte autoval limite}
\begin{cases}
-\Delta u = \lambda u & \text{in } \Omega, \\[5 pt]
\displaystyle\dfrac{\partial u}{\partial \nu_0} + \dfrac{1}{h} u = 0 & \text{on } \partial \Omega,
\end{cases}
\end{equation}
which admits a discrete spectrum 
\[0<\lambda^1(h)\le\lambda^2(h)\le\dots\le\lambda^j(h)\le\dots\to+\infty.\]
We remark that the boundary condition in \eqref{eqD forte autoval limite} can be formally obtained by the one in \eqref{eq forte autoval limite} setting $\beta=+\infty$. We show that, in light of the classical results by \cite{BCF80} and \cite{AB86}, \autoref{cor3.6} and \autoref{asymptotic} still holds true formally setting $\beta=+\infty$ in the formulas. However, we are not able to fully recover \autoref{optlambda}. 

\subsection*{Convergence of the spectrum}

As in \autoref{spectr}, all $\eps>0$ and $f\in L^2(\Om_\eps)$ we consider $v_{\eps,h, f}=\mathcal{T}_{\eps,h} f\in H^1_0(\Om_\eps)$ the unique weak solution of
\[
\begin{cases}
-\Delta v_{\varepsilon,h,f} = f & \text{in } \Omega, \\[5 pt]
-\eps\Delta v_{\varepsilon,h,f} = f & \text{in } \Sigma_\eps, \\[5pt]
v_{\eps, h,f}^-=v_{\eps,h,f}^+&\text{on } \partial \Om, \\[5 pt]
\displaystyle\dfrac{\partial {v_{\eps, h,f}^-}}{\partial \nu_0} = \varepsilon \dfrac{\partial v_{\eps, h,f}^+}{\partial \nu_0} & \text{on } \partial \Omega,\\[10 pt]
 v_{\varepsilon,h, f} = 0 & \text{on } \partial \Omega_\varepsilon. 
\end{cases}
\]
which again, is the unique solution of
\begin{equation}\label{eqD forte}\int_{\Omega} \nabla v_{\eps,h,f}\nabla \varphi\,dx+\eps\int_{\Sigma_\eps} \nabla v_{\eps,h,f}\nabla \varphi\,dx=\int_{\partial\Omega_\eps}f\varphi\,dx,\end{equation}
for every $\varphi\in H^1_0(\Omega_\eps)$.
Moreover, $v_{\eps,h,f}$ is also the unique minimiser in $H^1_0(\Om_\eps)$ of the functional
\begin{equation}\label{funzionaleD Gepsf}
    G_{\eps,f}(w):=\int_{\Om_\eps}a(d)\abs{\nabla w}^2 \,dx- 2\int_{\Om_\eps} w f\,dx.
\end{equation}

For all $f\in L^2(\Om)$, consider $v_{h,f}=\mathcal{T}_hf\in H^1(\Om)$ the unique weak solution of the boundary value problem 
\[
\begin{cases}
-\Delta v_{h,f} = f & \text{in } \Omega, \\[5 pt]
\displaystyle\dfrac{\partial v_{h,f}}{\partial \nu_0} + \dfrac{1}{ h} v_{h,f} = 0 & \text{on } \partial \Omega.
\end{cases}
\]
That is the unique solution of
\begin{equation}\label{eqD forte limite}\int_\Omega\nabla v_{h,f}\nabla\varphi\,dx+\int_{\partial\Omega}\dfrac{v_{h,f}\varphi}{ h}\,d\Hn=\int_\Om f\varphi\,dx,\end{equation}
for every $\varphi\in H^1(\Omega)$.
$v_{h,f}$ is also the unique minimiser in $H^1(\Om)$ of the functional
\begin{equation}\label{GD 0 f}
G_{f}(w):=\int_{\Om}|\nabla w|^2 \,dx+\int_{\partial\Om} \dfrac{w^2}{h} \,d\Hn- 2\int_{\Om} w f\,dx.
\end{equation}

Again, for $\eps$, the operator $ \mathcal{T}_{\eps,h} : L^2(\Omega_\varepsilon) \to L^2(\Omega_\varepsilon) $, is positive, linear, continuous, compact, and self-adjoint. Similarly, as $h^{-1}$ is bounded and strictly positive, $ \mathcal{T}_h : L^2(\Omega) \to L^2(\Omega) $ satisfies the same properties. In order to apply \autoref{convergenza spettri astratto} we need to show that $\mathcal{T}_{\eps,h}$ and $\mathcal{T}_h$ satisfy assumptions \Href{H2}--\Href{H4} (assumption \Href{H1} is trivially satisfied setting $\mathcal{V}=L^2(\Om)$ and $R_\eps$ the extension to $0$ outside $\Om$). As in \autoref{ipotesi conv astratta}, is enough to have an uniform Poincaré-type inequality and the $\Gamma$-convergence of the functional $G_{\eps,f}$ to $G_{f}$ in $L^2(\R^n)$. The latter has been famously proved in \cite{AB86}, while the former can be find in the proof of Lemma 1.1 of the seminar paper \cite{BCF80}, and are namely equations (1.14) and (1.13), which we state below. Notice that the aforementioned inequalities are stated for solutions to \eqref{eqD forte} but are true for any function in $H^1_0(\Om_\eps)$.
\begin{lemma}\label{lemma:Dpoinc}
Let $h\colon\partial\Omega\to\R$ be a positive  Lipschitz function . Then there exist positive constants $d_0=d_0(\Omega)$, and $C_p=C_p(\Omega,\norma{h}_{C^{0,1}})$ such that if
\begin{equation}
    \eps\norma{h}_{\infty}\le d_0,
\end{equation}
then for every $w\in H^1_0(\Omega_\eps)$
\begin{equation}\label{eq:Depspoinc}\int_{\Omega_\eps} w^2\,dx\le C_p\int_{\Omega_\eps} a(d) \abs{\nabla w}^2\,dx,\end{equation}
and
\begin{equation}\label{eq:Depspoinc2}\int_{\Sigma_\eps} w^2\,dx\le \eps^2 C_p\int_{\Sigma_\eps} \abs{\nabla w}^2\,dx.\end{equation}
\end{lemma}
Thus, \autoref{cor3.6} also holds true in this case. 

\subsection*{Optimisation problems}

We start by noticing that for a general, non-negative, function $h\in L^1(\partial\Om)$, the Poisson problem
\begin{equation}\label{eqDP}
    \int_\Om \nabla v\nabla\varphi\,dx+\int_\Om\dfrac{v\varphi}{h}\,d\Hn=\int_\Om f\varphi,
\end{equation}
with $f\in L^2(\Om)$, it may not be well defined in $H^1(\Omega)$ unless we require additional assumptions on $h$, such as $h^{-1}\in L^p(\partial\Om)$ with $p>n-1$ (see, for instance, \cite[\S3.6]{LU68}).\medskip

A simple solution to such a problem is to restrict our attention to the functions that satisfy a uniform lower bound. Fix $c_1>0$, then we consider
\[\mathcal{H}_m(\partial\Om, c_1):=\Set{h\in L^1(\Om): h\geq c_1,\,\text{ and }\,\int_{\partial\Om}(h-c_1)\,d\Hn\leq m}.\]
As in \autoref{opt}, we say that a sequence  $\set{h_k}$ converges in $\mathcal{H}_m(\partial\Om, c_1)$ to a function $h$ if and only if $h_k^{-1}$ converges to $h^{-1}$ weakly-* in $L^{\infty}(\partial\Om)$. With this topology $\mathcal{H}_m(\partial \Om,c_1)$ is than compact. The following Poincaré-type inequality holds in $\mathcal{H}_m$, which is a consequence of the optimisation of the principal eigenvalue, $\lambda^1(h)$, of \eqref{eqD forte autoval limite} proved in \cite{BBN17}. 
\begin{lemma}
    Fix $m>0$. There exists a constant $C=C(\Om,m)$ such that for all $h\in\mathcal{H}_m$ and all $w\in H^1(\Om)$ 
    \[\int_\Om w^2\,dx\le C \left[\int_\Om \abs{\nabla w}^2\,dx+\int_{\partial\Om}\dfrac{w^2}{h}\,d\Hn\right].\]
\end{lemma}
Then, as in \autoref{opt} the solutions to the Poisson problem \eqref{eqDP} as well as the eigenvalues of \eqref{eqD forte autoval limite} are continuous in $\mathcal{H}_m(\partial\Om,c_1)$, in particular we have that the optimisation problem
\[\min_{h\in\mathcal{H}_m(\partial\Om,c_1)}F(\lambda^{j_1}(h),\lambda^{j_2}(h),\dots,\lambda^{j_r}(h))\]
admits a solution for every lower semicontinuous function $F\colon\R^r\to\R$ and indices $j_1,j_2,\dots,j_r\in\N$.

\subsection*{Asymptotic development}
Fix $h\in C^{1,1}(\Gamma_{d_0})$. As in \autoref{dev}, let $\lambda$ be an eigenvalue of multiplicity $l+1$ for the limit eigenvalue problem \eqref{eqD forte autoval limite}, that is
\[\lambda=\lambda^i(h)=\dots=\lambda^{i+l}(h)\]
for some $i\in\N$, and denote by $E(\lambda)$ the associated eigenspace in $H^1(\Omega)$. We want to apply the lemma on small eigenvalues to the quadratic form on $H^1_0(\Om_\eps)$
\[q_{\lambda}(v,w)=\int_{\Om_\eps}a(d)\nabla v\nabla w\,dx-\lambda\int_{\Om_\eps}vw\,dx,\]
to compute
\[\lim_{\eps\to0^+}\dfrac{\lambda_\eps^{i+j}(h)-\lambda^{i+j}(h)}{\eps},\]
for $j=0,\dots,l$. For every $j\in\N$ the $j$-th eigenvalues of $q_\lambda$  is exactly $\eta^j_\eps=\lambda^j_\eps(h)-\lambda^j(h)$, hence $\eta^j_\eps$ converges to zero, as $\eps$ goes to zero, if and only if $j=i,\dots,i+l$, so that assumption \Href{C1}, in the Lemma on small eigenvalue, is verified for $\eps$ sufficiently small.\medskip

In analogy to the Robin case, for every $\eps>0$, and $u\in E(\lambda)$ we define
\[\hat{u}_\eps(x):=\begin{cases}
    u(x) &\text{if }x\in \bar{\Omega},\\[5pt]
    u(\sigma(x))\left(1-\dfrac{d(x)}{\eps h(x)}\right) &\text{if } x\in\Sigma_\eps,\\[5pt]
    0 &\text{if }x \in \R^n\setminus\Om_\eps,
\end{cases}\]
and consider the $l+1$ dimensional subspace of $H^1_0(\Om_\eps)$ defined as
\[F:=\Set{\mathcal{T}_{\eps,h}\hat{u}_\eps \colon u\in E(\lambda)}.\]

For every $v\in F$ there exists $u\in E(\lambda)$ such that $v=\mathcal{T}_{\eps,h}(\lambda\hat{u}_\eps)$ and
\[\lim_{\eps\to0^+}\norma{\mathcal{T}_{\eps,h}(\lambda\hat{u}_\eps) - \hat{u}_\eps}_{L^2(\Omega_\eps)}=0.\]
To prove \autoref{asymptotic} we will need the following estimate. For every $u\in E(\lambda)$ with $\norma{u}_{L^2(\Om_\eps)}=1$
 \begin{equation}\label{5.1strongD}
      \int_{\Om_\eps} a(d) \abs{\nabla (\mathcal{T}_{\eps,h}(\lambda\hat{u}_\eps)-\hat{u}_\eps)}^2\,dx=o(\eps).
   \end{equation}

As for the Robin case, such an estimate in particular, implies the existence of $\delta=\delta(\eps)$, with $\delta^2=o(\eps)$, such that for every $u\in E(\lambda)$,
\begin{equation}\label{selD1}
    \norma{\mathcal{T}_{\eps,h}(\lambda\hat{u}_\eps) - \hat{u}_\eps}_{L^2(\Omega_\eps)}\le \norma{u}_{L^2(\Om)}\delta. 
\end{equation}
By \eqref{selD1}, for every $w\in H^1_0(\Om_\eps)$ and $v\in F$, we have 
\[\dfrac{q_\lambda(v,w)}{\norma{v}_{L^2(\Om_\eps)} \norma{w}_{L^2(\Om_\eps)} }\le \lambda  \dfrac{\norma{u}_{L^2(\Om)} }{\norma{\mathcal{T}_{\eps,h}(\lambda\hat{u}_\eps)}_{L^2(\Om_\eps)} }\delta \le C\delta,\]
for some $u\in E(\lambda)$. Hence,  for $\eps$ sufficiently small, \eqref{5.1strongD} implies the validity of assumption \Href{C2}. Finally, and using the lemma on small eigenvalues, we have that for every $j=0,\dots,l$
\[\lim_{\eps\to0^+}\dfrac{\lambda_\eps^{i+j}(h)}{\eps}=\lim_{\eps\to0^+}\dfrac{\xi_\eps^j}{\eps}\]
so that we are left to prove that
\[\lim_{\eps\to0^+}\dfrac{\xi_\eps^j}{\eps}=\zeta_\lambda^{j}(h),\]
 the $(j+1)$-th eigenvalue of the bilinear form 
\begin{equation}\label{eqDzeta}Q_\lambda(u,v) =
\dfrac{1}{2}\int_{\partial \Om}Huv\,d\Hn -\dfrac{1}{3}\lambda\int_{\partial\Om}huv\,d\Hn,\end{equation}
on the $(l+1)$-dimensional space $E(\lambda)$.\medskip

Estimate \eqref{5.1strongD} and equation \eqref{eqDzeta} follows from the same computations as in \autoref{lemma 5.1strong} and \autoref{xitozeta} provided that we prove the following proposition.

\begin{prop}\label{teorACepsD}
Let $h\in C^{1,1}(\Gamma_{d_0})$ be a positive function such that $h(x)=h(\sigma(x))$. Let $\mathcal{G}=\set{f_\eps}\subset L^2(\Omega_\eps)$ and assume that there exist a constant $C_\mathcal{G}>0$ such that for every $\eps>0$
\begin{equation*}
 \int_{\Omega_\eps} f_\eps^2\,dx\le C_\mathcal{G},
\end{equation*}
and
\begin{equation*}
   \int_{\Sigma_\eps}f_\eps^2\,dx\le \eps C_\mathcal{G}.
\end{equation*}
Let $v_\eps=v_{\eps,f_\eps}$ and 
    \[
    \tilde{v}_\eps(z)=\begin{cases}
        v_\eps(z) &\text{if }z\in\Omega,\\[5 pt]
        v_\eps(\Psi_\eps(z)) &\text{if }z\in\Sigma_1.
    \end{cases}
    \] 
 Then the family $\set{\tilde{v}_\eps}$ is equibounded in $H^1(\Omega_1)$. Moreover, up to a subsequence, $\set {f_\eps}$ converges weakly in $L^2$ to a function $f$ with $f=0$ almost everywhere in $\R^n\setminus\Omega$ and, up to another subsequence, $\set{\tilde{v}_\eps}$ converges weakly in $H^1(\Omega_1)$, as $\eps$ goes to $0$, to the function 
    \begin{equation*}\tilde{v}_f(z)=\begin{cases} 
v_f(z) &\text{if }z\in\Omega,\\[5 pt]
    v_f(\sigma(z))\left(1-\dfrac{d(z)}{h(z)}\right) &\text{if }z\in\Sigma_1.
\end{cases}\end{equation*}
\end{prop}
\begin{proof}
The proof follows the same scheme as the one of \autoref{teorAC}, hence we limit ourself to point out the meaningful differences.

In the first step, one assumes that the family $\set{\tilde{v}_\eps}$ converges weakly in $H^1_0(\Omega_1)$ to a function $\tilde{v}$. Then $\tilde{v}$, in $\Sigma_1$, satisfies the equation 
\begin{equation}\label{eqlimweakD}
\int_{\Sigma_1} (\nabla\tilde{v}\cdot\nu_0)(\nabla\varphi\cdot\nu_0)J_0(z)\,dz=0, \end{equation}
for every $\varphi\in H^1_0(\Sigma_1)$, where 
\[J_0(z)=\prod_{i=1}^{n-1}\dfrac{1}{1+h(z)k_i(\sigma(z))},\]
and $k_i$ are the principal curvatures of $\partial\Omega$. The solution to \eqref{eqlimweakD} is uniquely determined by the Dirichlet boundary datum on $\partial\Omega$. Namely, for every $w\in H^1(\Sigma_1)$, the solution to \eqref{eqlimweakD} with $\tilde{v}=w$ on $\partial\Omega$ is 
\[\tilde{v}(z)=w(\sigma(z))\left(1-\dfrac{d(z)}{h(z)}\right).\]

The second step is to use the uniform $H^1$ and $H^2$ density estimates 
\begin{equation}\label{BCFpoinc}\int_{\Omega_\eps}a(d)\abs{\nabla v_\eps}^2\,dx+\int_{\Omega_\eps}a(d)\abs{D^2 v_\eps}^2\,dx\le C(\Omega,h,C_\mathcal{G}))\end{equation}
proved in \cite[Lemma 1.1]{BCF80}, to show the equiboundedness in $H^1_0(\Omega_1)$ of the family $\set{\tilde{v}_\eps}$. Indeed, by direct computation
\begin{equation}\label{dircomp}\int_{\Sigma_1}\abs{\nabla \tilde{v}_\eps}^2\,dz\le C\left[\eps^{-1}\int_{\Sigma_\eps}\abs{\nabla^{\partial\Om_\eps} v_\eps}^2\,dx+\eps\int_{\Sigma_\eps}\abs{\nabla v_\eps}^2\,dx\right].\end{equation}
Using the Dirichlet boundary condition, the gradient of $v_\eps$ on $\partial\Om_\eps$ is normal to the boundary, so that $\nabla^{\partial\Om_\eps}v_\eps|_{\partial\Om_\eps}=0$ and, we can use \eqref{eq:Depspoinc2} to estimate its $L^2$ norm on $\Sigma_\eps$ as
\begin{equation}\label{nablasmall}\int_{\Sigma_\eps}\abs{\nabla^{\partial\Om_\eps} v_\eps}^2\,dx\le\eps^2C_p \int_{\Sigma_\eps}\abs{D^2 v_\eps}^2\,dx.\end{equation}
putting together \eqref{dircomp}, \eqref{nablasmall} and \eqref{BCFpoinc} we have
\[\int_{\Sigma_1}\abs{\nabla \tilde{v}_\eps}^2\,dz\le C(\Omega,h,C_\mathcal{G}).\]
The third and last step, is to notice that, if $f_{\eps_k}$ converges weakly to $f$ in $L^2$, then $v_{\eps_k}$ converges, weakly in $H^1(\Omega)$ to $v_f$. Hence, by step one and step two, up to a further subsequence $\tilde{v}_{\eps_k}$ converges weakly $H^1_0(\Om_1)$ to $\tilde{v}_f$.
\end{proof}

\section*{Appendix}
\addcontentsline{toc}{section}{Appendix}
\numberwithin{equation}{section}
\renewcommand{\thesection}{A}    
\setcounter{teor}{0}
\setcounter{equation}{0}
For the sake of completeness, in this section, we include the proof of the energy estimates used in \autoref{dev}.

\label{energy}
We now prove uniform $H^2$ for the functions $v_\eps$, solution to the weak equation 
\begin{equation}\label{weps}
\int_{\Omega_\eps} a(d)\nabla v_\eps\nabla\varphi\,dx+\beta\int_{\partial\Omega_\eps} v_\eps\varphi\,d\Hn = \int_{\Omega_\eps} f_\eps\varphi\,dx,
\end{equation}
for every $\varphi\in H^1(\Omega_\eps)$, where we recall 
\[
a(t) = \begin{cases} 
\eps & \text{if } t > 0, \\
1 & \text{if } t \leq 0.
\end{cases}
\]

 In the following, we will always assume that $\norma{h}_\infty < d_0$ and $\eps<1$, so that the distance function $d\in C^{1,1}(\Sigma_\eps)$. Let 
\[\nabla^{\partial \Om_\eps} v_\eps = \nabla v_\eps - \left(\nabla v_\eps \cdot \nu_\eps\right)\nu_\eps.\]
We prove \autoref{teor: C11energy}, of which we recall the statement.

\COneEnergy*

The regularity of $v_\eps$ is classical and well known (see, for instance, \cite[\S 4.16]{LU68}). We start by noticing that the $H^1$ energy estimate
\begin{equation}\label{H1est}
    \int_{\Om_\eps}a(d)\abs{\nabla v_\eps}^2\,dx+\beta\int_{\partial\Om_\eps}v_\eps^2\,d\Hn\le C.
\end{equation}
holds, indeed, in the notation of \autoref{spectr}, $v_\eps=\mathcal{T}_{\eps,h} f_\eps$,  hence by \eqref{boundedness norma X}
\[\int_{\Om_\eps}a(d)\abs{\nabla v_\eps}^2\,dx+\beta\int_{\partial\Om_\eps}v_\eps^2\,d\Hn\le C \norma{f_\eps}_{L^2(\Omega_\eps)}^2\]
which together with \eqref{fespbound}, gives the claim.\medskip

Following \cite{BCF80,F80}, we locally represent $\partial \Omega$ and $\partial \Omega_\eps$ uniformly with respect to $\eps$. This allows us to construct a local diffeomorphism that maps $\partial \Omega$ and $\partial \Omega_\eps$ into subsets of two parallel hyperplanes. For every $\delta > 0$ we define the cube
\[
\tilde{Q}_{\delta} = \{ y \in \mathbb{R}^n \mid |y_i| \leq \delta, \text{ for every } 1 \leq i \leq n \}.
\]
We recall the following lemma \cite[Lemma 5.3]{AC25}
\begin{lemma}\label{lemma flattering}
Let  $h \in C^{1,1}(\Gamma_{d_0})$ be a positive function such that $h(x)=h(\sigma(x))$, and let $\sigma_0 \in \partial \Omega$. There exists $\eps_0 = \eps_0(\Omega, \sigma_0)$ such that, if
\[
\eps \| h \|_{C^{0,1}} < \eps_0,
\]
then there exists an open set $V$ containing $\sigma_0$ and $\sigma_0 + \eps h(\sigma_0)\nu_0(\sigma_0)$, and there exist functions $g, k_\eps: \mathbb{R}^{n-1} \to \mathbb{R}$ such that, up to a rototranslation,
\[
\begin{aligned}
\Omega \cap V &= \left\{ (x', x_n) \,\middle|\, x_n \leq g(x') \right\} \cap V, \\
\Omega_\eps \cap V &= \left\{ (x', x_n) \,\middle|\, x_n \leq g(x') + \eps k_\eps(x') \right\} \cap V,
\end{aligned}
\]
and
\[
\begin{aligned}
\partial \Omega \cap V &= \left\{ (x', x_n) \,\middle|\, x_n = g(x') \right\} \cap V, \\
\partial \Omega_\eps \cap V &= \left\{ (x', x_n) \,\middle|\, x_n = g(x') + \eps k_\eps(x') \right\} \cap V.
\end{aligned}
\]
Moreover, the map 
\[
\Phi_{\sigma_0} : (x', x_n) \in V \mapsto \left( x', \dfrac{x_n - g(x')}{k_{\eps}(x')} \right) \in \tilde{Q}_{\delta_0},
\]
for some $\delta_0 = \delta_0(\sigma_0, \Omega) > 0$, is invertible and  there exists a positive constant $C=C(\Omega,\norma{h}_{C^{1,1}})$ such that 
    \begin{equation}
        \label{eq: kepsstime}
        \min k_\eps\ge \dfrac{1}{C}, \qquad\qquad  \norma{k_\eps}_{C^{1,1}}+\norma{\Phi}_{C^{1,1}}+\norma{\Phi^{-1}}_{C^{1,1}}\le C.
    \end{equation}
\end{lemma}
For simplicity's sake, we will drop the explicit dependence on the point $\sigma_0 \in \partial \Omega$ when possible. Notice that the diffeomorphism $\Phi$ flattens  $\partial\Om$ and $\partial\Om_\eps$, in the sense that
\[
\Phi(\partial \Omega \cap V) = \{ y_n = 0 \} \cap \tilde{Q}_{\delta_0},\quad \text{and}\quad  \Phi(\partial \Omega_{\eps} \cap V) = \{ y_n = \eps \} \cap \tilde{Q}_{\delta_0}.
\]
 For every $0 < \delta < \delta_0$ we define
\[
Q_{\delta} = \Phi^{-1}(\tilde{Q}_{\delta}).
\]

\begin{oss}\label{defveps}
    Consider $v_\eps$ the solution to equation \eqref{weps}, fix \( \sigma_0 \in \partial \Omega \), and define \( w_\eps(y): = v_\eps(\Phi^{-1}(y)) \). Then, we have:
\[
w_\eps \in H^1(\{ y_n < \eps \} \cap \tilde{Q}_{\delta_0}) \cap H^2((\{ y_n < \eps \} \cap \tilde{Q}_{\delta_0}) \setminus \{ y_n = 0 \}),
\]
and, it solves the equation:
\begin{equation}\label{eqveps}
\int_{\{ y_n < \eps \} \cap \tilde{Q}_{\delta_0}} a(y_n) A_{\eps} \nabla w_\eps \cdot \nabla \varphi\, dy + \beta \int_{\{ y_n = \eps \} \cap \tilde{Q}_{\delta_0}} w_\eps \varphi J_{\eps} \, d\mathcal{H}^{n-1} =\int_{\{ y_n < \eps \} \cap \tilde{Q}_{\delta_0}} p_\eps\varphi \, dy,
\end{equation}
for all \( \varphi \in H^1_0(\tilde{Q}_{\delta_0}) \), where
\[
A_{\eps}(y) = k_{\eps}(y') (D(\Phi^{-1})(y))^{-1} (D(\Phi^{-1})(y))^{-T},\quad 
 J_{\eps}(y) = \sqrt{1 + \abs{\nabla g(y')}^2 + \eps \abs{\nabla k_{\eps}(y')}^2},
\]
and
\[p_\eps(y)=f_\eps(\Phi^{-1}(y))k_\eps(y').\]

Notice that $A_\eps$ is elliptic and bounded, uniformly in $y$ and $\eps$. Using \eqref{H1est} we have that there exists a constant $C=C(\Om,h,C_\mathcal{G},\beta,\sigma_0)$ such that

\begin{equation}\label{vH1est}
   \int_{\{ y_n < \eps \} \cap \tilde{Q}_{\delta_0}} a(y_n) \abs{ \nabla w_\eps }^2 \, dy + \beta \int_{\{ y_n = \eps \} \cap \tilde{Q}_{\delta_0}} w_\eps^2 \, d\mathcal{H}^{n-1} \leq C. 
\end{equation}
\end{oss}

Since $w_\eps\in H^2((\{ y_n < \eps \} \cap \tilde{Q}_{\delta_0}) \setminus \{ y_n = 0 \})$ and $v_\eps\in H^2((\Om_\eps\cap V)\setminus\partial V)$, for every $\delta\in(0,\delta_0)$, we can consider 
\[\Tilde{I}_{\delta,\sigma_0}:=\int_{\{y_n<\eps\}\cap \tilde{Q}_\delta}a(y_n)\abs{D^2w_\eps}^2dy+\beta\int_{\{y_n=\eps\}\cap \tilde{Q}_\delta}\abs{\nabla_{n-1}w_\eps}^2d\Hn,\]
and
\[I_{\delta,\sigma_0}:=\int_{\Om_\eps\cap Q_\delta}a(d)\abs{D^2v_\eps}^2dx+\beta\int_{\partial \Om_\eps\cap {Q}_\delta}\abs{\nabla^{\partial\Om_\eps}v_\eps}^2d\Hn,\]
where 
\[\nabla_{n-1}w_\eps=\nabla w_\eps-\dfrac{\partial w_\eps}{\partial n}\textbf{e}_n\quad \text{and}\quad \nabla^{\partial\Om_\eps}v_\eps=\nabla v_\eps-(\nabla u_\eps\cdot \nu_\eps)\nu_\eps. \]
For simplicity's sake, we will drop the explicit dependence on the point $\sigma_0 \in \partial \Omega$ when possible.
\begin{oss}\label{equivalbound}
We aim to bound $I_\delta$ uniformly. From \cite[Lemma 5.6]{AC25} we have that

\[
I_{\delta} \leq C \left( \tilde{I}_{\delta} + \int_{\{y_n < \epsilon\} \cap \tilde{Q}_{\delta}} \epsilon(y_n) |\nabla w_\eps|^2 \, dy \right),
\]
and
\[
\tilde{I}_{\delta} \leq C \left( I_{\delta}+ \int_{\Omega_{\epsilon} \cap Q_{\delta}} \epsilon(d) |\nabla v_\eps|^2 \, dx \right),
\]
for some positive constant $C=C(\Om,\norma{h}_{C^{1,1}},\sigma_0)$.
Then, using \eqref{H1est} and \eqref{vH1est},  $I_{\delta}$  is bounded if and only if $\tilde{I}_{\delta}$ is bounded.
\end{oss}

\begin{lemma}\label{H2estloc}
   Let $h\in C^{1,1}(\Gamma_{d_0})$ be a positive function such that $\eps \norma{h}_{C^{1,0}}\leq \eps_0$ and $h(x)=h(\sigma(x))$. If $\sigma_0\in \partial \Om$ and  $w_\eps$ is as in \autoref{defveps}, then there exists $C=C(\Om,\norma{h}_{C^{1,1}},C_\mathcal{G},\beta, \sigma_0)$ such that 
    \[\tilde{I}_{{\delta_0}/{2}}\leq C\]
\end{lemma}
\begin{proof}
     Let $\xi \in C_c^\infty(\tilde{Q}_{\delta_0})$ be a non-negative function such that $\xi \equiv 1$ in $\tilde{Q}_{\delta_0/2}$. Choosing $|\eta|$ small enough, we have
\[
\text{supp} \xi + \eta e_i \subset \tilde{Q}_{\delta_0}
\]
for every $k = 1, \ldots, n - 1$.  In the following, for any $L^2$ function $\psi$, we denote by
\[
\Delta_k^\eta \psi(y) = \dfrac{\psi(y + \eta e_k) - \psi(y)}{\eta}
\]
the difference quotients (see \cite[\S5.8.2]{E10}). For any couple of functions $\psi_1$ and $\psi_2$ we recall that
\begin{equation}\label{prodiff}
\Delta_k^\eta (\psi_1 \psi_2)(y) = \Delta_k^\eta \psi_1(y) \psi_2(y) + \psi_1(y + \eta e_k) \Delta_k^\eta \psi_2(y),
\end{equation}
moreover, if $\psi_1$ and $\psi_2$ are measurable and
\begin{equation}\label{eqsupp}
(\text{supp} \psi_1 \cap \text{supp} \psi_2) \pm \eta e_k \subset \tilde{Q}_\delta,
\end{equation}
then, for every $k = 1, \ldots, n - 1$, it holds
\begin{equation}\label{partidifquo}
    \int_{\tilde{Q}_{\delta_0}} \psi_1(y) \Delta_k^\eta \psi_2(y) \, dy = -\int_{\tilde{Q}_{\delta_0}} \Delta_k^{-\eta} \psi_1(y) \psi_2(y) \, dy.
\end{equation}
Finally, if $U\subset \tilde{Q}_{\delta_0}$ such that $U\pm \eta e_k \subset \tilde{Q}_{\delta_0}$ and $\psi\in H^{1}(\{y_n<\eps\}\cap \tilde{Q}_{\delta_0})$ then we have 
\begin{equation}\label{diffeder}
   \int_U(\Delta^\eta_k\psi)^2\leq\int_{\tilde{Q}_{\delta_0}}|\partial_k \psi|^2. 
\end{equation}
Let \[\varphi=-\Delta^{-\eta}_k(\xi^2\Delta^\eta_kw_\eps).\] For $\abs{\eta}$ small enough, we can use  $\varphi$ as a test function in equation \eqref{eqveps}. Let 
\[\mathcal{I} := \int_{\{ y_n < \eps \} \cap \tilde{Q}_{\delta_0}} a(y_n) A_{\eps} \nabla w_\eps \cdot \nabla \varphi \, dy + \beta \int_{\{ y_n = \eps \} \cap \tilde{Q}_{\delta_0}} w_\eps \varphi J_{\eps} \, d\mathcal{H}^{n-1} \]
be the left-hand side of the equation
and let
\[\mathcal{J} := \int_{\{ y_n < \eps \} \cap \tilde{Q}_{\delta_0}}p_\eps\varphi \, dy\]
be the right-hand side.

Using standard elliptic regularity arguments (see \cite[Lemma 5.7]{AC25} for the details), there exists a constant $C=C(\Om,\norma{h}_{C^{1,1}}, \sigma_0)>0$ such that
\begin{equation}\label{dismathI}
    \begin{split}
        \mathcal{I}\geq \dfrac{C_1}{2} \int_{\{y_n < \eps\} \cap \tilde{Q}_{\delta_0}} a(y_n) \xi^2 \sum_{j=1}^{n} \left( \Delta_k^{\eta} (\partial_j w_\eps) \right)^2 \, dy - C \int_{\{y_n < \eps\}} a(y_n) |\nabla w_\eps|^2 \, dy \\
        + C\beta\int_{\{y_n = \eps\} \cap \tilde{Q}_{\delta_0}} \xi^2 (\Delta_k^\eta w_\eps)^2 \, d\Hn - C \beta \int_{\{y_n = \eps\} \cap \tilde{Q}_{\delta_0}} w_\eps^2 \, d\Hn.
    \end{split}
\end{equation}
where $C_1$ bounds from below the ellipticity constant of $A$. \\
Moreover, using Young's inequality
\[ab\le \dfrac{\gamma}{2}a^2+\dfrac{1}{2\gamma}b^2\]
with $\gamma=\dfrac{C_1a(y_n)}{4}$ we have that 
\[\mathcal{{J}}\leq \dfrac{C_1}{8}\int_{\{ y_n < \eps \} \cap \tilde{Q}_{\delta_0}}a(y_n)(\Delta^{-\eta}_k(\xi^2\Delta^{\eta}_kw_\eps))^2dy+\dfrac{2}{C_1}\int_{\{ y_n < \eps \} \cap \tilde{Q}_{\delta_0}}\dfrac{1}{a(y_n)}p_\eps^2dy \]
Using \eqref{fespbound} and \eqref{fespconc}, we obtain that for some positive constant $C=C(\Om,\norma{h}_{C^{1,1}},C_\mathcal{G},\beta, \sigma_0)$
\begin{equation}\label{vnonconc}\dfrac{2}{C_1}\int_{\{ y_n < \eps \} \cap \tilde{Q}_{\delta_0}}\dfrac{1}{a(y_n)}p_\eps^2\,dy\leq C\end{equation}
As a consequence, applying \eqref{diffeder} and Young's inequality, we have that 
\begin{equation}\label{dismathJ}
\begin{split}
    \mathcal{J}&\leq C+\dfrac{C_1}{8}\int_{\{ y_n \leq \eps \} \cap \tilde{Q}_{\delta_0}}a(y_n)(\partial_k(\xi^2\Delta_k^\eta w_\eps))^2\,dy\\
    &\leq C+\dfrac{C_1}{4}\int_{\{ y_n < \eps \} \cap \tilde{Q}_{\delta_0}} a(y_n)\big[\big(2\xi(\partial_k\xi)\Delta^\eta_kw_\eps\big)^2+\xi^4(\Delta^\eta_k\partial_kw_\eps)^2\big]\,dy\\
    &\leq C+C\int_{\{ y_n < \eps \} \cap \tilde{Q}_{\delta_0}} a(y_n)\abs{\nabla w_\eps}^2\,dy+\dfrac{C_1}{4}\int_{\{ y_n < \eps \} \cap \tilde{Q}_{\delta_0}} a(y_n) \xi^2\sum_{j=1}^n(\Delta^\eta_k\partial_j w_\eps)^2\,dy.
\end{split}
\end{equation}
Combining \eqref{eqveps}, \eqref{dismathI} and \eqref{dismathJ} and using that $\xi=1$ on $\tilde{Q}_{\delta_0/2}$ we have that 
\begin{equation}
    \begin{split}
 \mathcal{K}&:=\dfrac{C_1}{4} \int_{\{y_n < \epsilon\} \cap \tilde{Q}_{\delta_0 / 2}} a(y_n) \sum_{j=1}^n (\Delta_k^\eta \partial_j w_\eps)^2 \, dy 
+ C \int_{\{y_n = \epsilon\} \cap \tilde{Q}_{\delta_0 / 2}} (\Delta_k^\eta w_\eps)^2 \, d\Hn \\
&\leq C     \bigg( \int_{\{y_n < \eps\} \cap \tilde{Q}_{\delta_0}} a(y_n) |\nabla w_\eps|^2 \, dy 
+ \beta \int_{\{y_n = \eps\} \cap \tilde{Q}_{\delta_0}} w_\eps^2 \, d\Hn +1\bigg).
    \end{split}
\end{equation}
Then, by \eqref{vH1est}, the right-hand side of the previous inequality is bounded, we have that for every $ k = 1, \ldots, n - 1,$  $\partial_kw_\eps\in H^1(\{y_n<\eps\}\cap \tilde{Q}_{\delta_0 / 2})$  and
\begin{equation}\label{H2esttan}
\int_{\{ y_n < \eps\} \cap \tilde{Q}_{\delta_0 / 2}} a(y_n) \sum_{j=1}^n (\partial_k \partial_j w_\eps)^2 \, dy +\int_{\{y_n = \epsilon\} \cap \tilde{Q}_{\delta_0 / 2}} \abs{\nabla_{n-1} w_\eps}^2 \, d\Hn \leq C
\end{equation}
We can now estimate $\partial_{nn}^2w_\eps$. From \eqref{eqveps}, we have that almost everywhere 
\[
     -\sum_{\substack{1\le i,j\le n \\ (i,j)\ne(n,n)}}a(y_n)\partial_i(a_{ij}\partial_{j}w_\eps)-a(y_n)\partial_n a_{nn}\partial_n w_\eps-a(y_n)a_{nn}\partial^2_{nn}w_\eps=p_\eps.
    \]
    From the uniform ellipticity of $A_\eps$, we have that $a_{nn}=A_\eps\mathbf{e}_n\cdot\mathbf{e}_n\ge C_1$, hence $\partial_{nn}^2w_\eps$ can be estimated in terms $A_\eps$, the other derivatives of $w_\eps$ and $p_{\eps}$. Namely we have 
    \begin{equation}\label{H2estnormal}\begin{split}
    \int_{\set{y_n<\eps}\cap \tilde{Q}_{\delta_0/2}} a(y_n) \abs{\partial_{nn}^2w_\eps}^2\,dy\le C\Bigg(&\int_{\set{y_n<\eps}\cap\tilde{Q}_{\delta_0/2}}a(y_n)\sum_{k=1}^{n-1}{\abs{\nabla \partial_k v}}^2\,dy \\[7pt] &+\int_{\set{y_n<\eps}\cap\tilde{Q}_{\delta_0/2}}a(y_n)\abs{\nabla v}^2\,d\Hn + \int_{\set{y_n<\eps}\cap \tilde{Q}_{\delta_0/2}} \dfrac{1}{a(y_n)}p_\eps^2\,dy \Bigg),
    \end{split}\end{equation}
    for some constant $C=C(\Om,h,\beta,\sigma_0)$. Finally, joining \eqref{H2esttan}, \eqref{H2estnormal}, \eqref{vH1est},  \eqref{fespbound}, and \eqref{fespconc} we have
    \[\tilde{I}_{\delta_0/2}\le C.\]
\end{proof}
We can now prove the following theorem.

\begin{proof}[Proof of \autoref{teor: C11energy}]
    Let $v_\eps$ be the solution to \eqref{weps}. Recall the definition of 
    \[I_{\delta,\sigma_0}=\int_{\Om_\eps\cap Q_\delta}a(d)\abs{D^2w_\eps}^2dx+\beta\int_{\partial \Om_\eps\cap {Q}_\delta}\abs{\nabla^{\partial\Om_\eps}w_\eps}^2d\Hn.\]
    From \autoref{equivalbound} and \autoref{H2estloc}, we have that for every $\sigma_0\in\partial\Omega$ there exists $\eps_0=\eps_0(\Om,\sigma_0)>0$, $\delta_0=\delta_0(\Om,\sigma_0)>0$ and a constant $C=C(\Om,h,\beta,\sigma_0)$, such that 
    \[I_{\delta_0/2,\sigma_0}\le C.\]
    From the boundedness of $\Om$ we have that there exists a finite number $\sigma_1,\sigma_2,\dots,\sigma_m\in\partial\Om$ with associated $\eps_i=\eps(\Om,\sigma_i)$, $\delta_i=\delta(\Om,\sigma_i)$  and $C_i=C(\Om,h,C_\mathcal{G},\beta,\sigma_i)$ such that 
    \[\Sigma_\eps\subset\bigcup_{i=1}^m \Phi^{-1}_{\sigma_i}(\tilde{Q}_{\delta_i/2})=:V_0\]
    for every $\eps<\eps_0(\Om):=\min\set{\eps_1,\eps_2,\dots,\eps_m}$ and
    \begin{equation}\label{locest}I_{\delta_i/2,\sigma_i}\le C_i,
    \end{equation}
    for every $i=1,2,\dots,m$.
    Let $U$ be an open set such that $U\subset\subset\Om$ and $\Om_\eps\subset U\cup V_0$. By standard elliptic regularity arguments, and by \eqref{fespbound}, there exists $C_0=C(\Om,C_\mathcal{G})>0$ such that
    \[\int_U\abs{D^2 v_\eps}^2\,dx\le C_0.\]
    Finally, summing the previous estimate to the estimates \eqref{locest}, we have 
    \[\int_{\Omega_\eps} a(d)\abs{D^2 v_\eps}^2\,dx+\beta\int_{\partial\Omega_\eps} \abs{\nabla^{\partial\Omega_\eps} v_\eps}^2\,d\Hn\le \sum_{i=0}^{m}C_i=:C(\Om,h,C_\mathcal{G},\beta).\]
\end{proof}

\subsubsection*{Acknowledgements} 
The authors would like to thank Roberto Ognibene for the helpful advice he provided.

The authors are members of Gruppo Nazionale per l’Analisi Matematica, la Probabilità e le loro Applicazioni
(GNAMPA) of Istituto Nazionale di Alta Matematica (INdAM). 

The author Emanuele Cristoforoni was partially supported by the INdAM - GNAMPA Project, 2025,
”Esistenza, unicità, simmetria e stabilità per problemi ellittici nonlineari e nonlocali”, CUP\_E5324001950001.


\subsubsection*{Statements and Declarations}
All authors declare that they have no conflicts of interest.

\printbibliography[heading=bibintoc]

\Addresses
\end{document}